\documentclass[12pt,twoside]{amsart}

\usepackage{amsmath}
\usepackage{amsfonts}
\usepackage{amssymb,enumerate}
\usepackage{amsthm}
\usepackage[all]{xy}
\usepackage{hyperref}
\allowdisplaybreaks[4] \footskip=12pt
\renewcommand{\uppercasenonmath}[1]{}

\numberwithin{equation}{section} \theoremstyle{plain}
\newtheorem*{thm*}{Main Theorem}
\newtheorem{thm}{Theorem}[section]
\newtheorem{cor}[thm]{Corollary}
\newtheorem*{cor*}{Corollary}
\newtheorem{lem}[thm]{Lemma}
\newtheorem*{lem*}{Lemma}

\newtheorem*{fact*}{Fact}

\newtheorem*{nota*}{Notation}
\newtheorem{prop}[thm]{Proposition}
\newtheorem*{prop*}{Proposition}
\newtheorem{rem}[thm]{Remark}
\newtheorem*{rem*}{Remark}

\newtheorem*{observation*}{Observation}

\newtheorem*{exa*}{Example}
\newtheorem{df}[thm]{Definition}
\newtheorem*{df*}{Definition}

\newtheorem*{conj*}{Conjecture}

\newtheorem*{con*}{Construction}

\renewcommand{\geq}{\geqslant}
\renewcommand{\leq}{\leqslant}

\begin{document}
\begin{center}
{\large  \bf  Finitistic dimensions in triangulated categories with a compact silting generator}

\vspace{0.5cm}  Xiaoyan Yang\\
%\bigskip
School of Science, Zhejiang University of Science and Technology, Hangzhou 310023\\
E-mail: yangxy@zust.edu.cn
\end{center}

\bigskip
\centerline { \bf  Abstract}
%\bigskip
\leftskip10truemm \rightskip10truemm \noindent This paper studies finitistic dimensions in triangulated categories with a compact silting generator. We unify several notions of (big) finitistic dimension appearing in the literature, show that they are essentially equivalent, and relate them to the abelian heart when the generator is tilting. This yields a new categorical perspective on the finitistic dimension conjecture. Furthermore,
we establish explicit inequalities for (big) finitistic and global dimensions under recollements, demonstrating that finiteness in the middle category is equivalent to finiteness in the outer categories. These results generalize classical ring-theoretic theorems and apply to triangular matrix rings, exact contexts and trivial extensions.\\
\vbox to 0.3cm{}

\leftskip0truemm \rightskip0truemm

\bigskip
\section{\bf Introduction}
In homological algebra and representation theory, global dimension and (big) finitistic dimension play crucial roles in understanding the structure and complexity of modules over algebras. A longstanding open problem is the \emph{finitistic dimension conjecture}, which asserts that every Artin algebra has finite finitistic dimension \cite{Ba}, where the finitistic dimension $\operatorname{fpd}\Lambda$ of an algebra $\Lambda$ is defined as the supremum of the projective dimensions of all
finitely generated modules having finite projective dimension. This conjecture implies several other homological conjectures, including the Auslander-Reiten conjecture, the Nakayama conjecture, the Wakamatsu tilting conjecture. Despite recent progress on  finitistic dimensions (\cite{CCZ,G,GI,Ri}),  the finitistic dimension conjecture remains open.

In light of the difficulty of directly settling this conjecture, researchers have shifted their focus from the classical module category to the more general framework of triangulated categories, aiming to reinterpret the finitistic dimension as an intrinsic property of such categories. To this end,
Krause \cite{K1} introduced a notion of \emph{finitistic dimension} for a Hom-finite triangulated category $\mathcal{S}$ via the generativity of objects, defined as
\[
\operatorname{fin.dim} \mathcal{S} := \inf \{\mathrm{sup}\{|n|\geq0\hspace{0.03cm}|\hspace{0.03cm}\mathrm{Hom}_\mathcal{T}(X,X[-n])\neq 0\}\hspace{0.03cm}|\hspace{0.03cm}X \text{ is a finitistic generator of } \mathcal{S} \},
\]and proved that for a ring $R$,
$\operatorname{fpd} R < \infty\Longleftrightarrow\operatorname{fin.dim} \operatorname{Perf}(R) < \infty$, where $\operatorname{Perf}(R)$ denote the category of perfect complexes over $R$.

A different notion, applicable to general triangulated categories, was later proposed by Biswas-Chen-Rahul-Parker-Zheng \cite{BCRPZ}. In their setting, finiteness serves as a necessary condition for the existence of a bounded $t$-structure and is related to the vanishing of the singularity category, thereby connecting to regularity conditions in algebraic geometry. This result vastly generalizes the conjecture of Antieau-Gepner-Heller on the regularity of schemes (see \cite[Conjecture 1.5]{AGH}). Building on these developments, Chen-Chen-Zhang \cite{CCZ} investigated (big) finitistic and global dimensions, establishing explicit inequalities that relate these dimensions of the middle category in a recollement of triangulated categories to those of the outer categories. More concretely, let $\mathcal{T}$ be a compactly generated triangulated category with a compact generator $M$ and $(\mathcal{T}^{\leq0},\mathcal{T}^{\geq0})$ the $t$-structure generated by $M$. Denote by $\mathcal{T}^\mathrm{b}$ and $\mathcal{T}^c$ the full subcategories
of $\mathcal{T}$ consisting of all bounded and compact objects, respectively.
The \emph{finitistic dimension} of $\mathcal{T}$ with respect to $M$ is defined as
$$\mathrm{fpd}(\mathcal{T},M):=\mathrm{sup}\{\mathrm{pdim}_M(X)
\hspace{0.03cm}|\hspace{0.03cm}
X\in\mathcal{T}^c\cap\mathcal{T}^\mathrm{b}\cap\mathcal{T}^{\geq0}\ \textrm{with}\ \mathrm{pdim}_M(X)<\infty\},$$where $\mathrm{pdim}_M(X):=\mathrm{inf}\{n\in\mathbb{Z}\hspace{0.03cm}|\hspace{0.03cm}
\mathrm{Hom}_\mathcal{T}(X,Y[i])=0\ \textrm{for\ all}\ Y\in\mathcal{T}^{\mathrm{b}}\cap\mathcal{T}^{\leq0}\ \textrm{and}\ i>n\}$. They also defined the \emph{finitistic dimension} of $\mathcal{T}^c$
at $M$ as
$$\mathrm{findim}(\mathcal{T}^c,M):=\mathrm{inf}\{n\in\mathbb{N}\hspace{0.03cm}|\hspace{0.03cm}
M(-\infty,-1]^{\bot}\cap\mathcal{T}^c \subseteq \langle M\rangle^{[0,+\infty)}[n]\},$$
where $M(-\infty,-1]^{\bot}=\{X\in\mathcal{T}\hspace{0.03cm}|\hspace{0.03cm}
\mathrm{Hom}_\mathcal{T}(M[n],X)=0\ \textrm{for}\ n\geq 1\}$.
Meanwhile, Kostas \cite{Ko} adopted a complementary perspective, defining the \emph{global dimension} $\mathrm{gl.dim}_{\mathcal{H}_M}\mathcal{T}$ of $\mathcal{T}$ relative to $\mathcal{H}_M:=\mathcal{T}^{\leq0}\cap\mathcal{T}^{\geq0}$ as the smallest $n$ so that $\mathrm{pdim}_M(X)\leq n$ for any $X\in\mathcal{H}_M$, and proved that the finiteness of this dimension is intrinsic to the triangulated category itself.

Taken together, these categorical perspectives not only unify the classical module-theoretic case but also extend naturally to geometric and derived settings, providing a broader framework for the study of homological dimensions. In this paper, we build on these insights and systematically investigates projective dimension, global dimension, and various (big) finitistic dimensions in triangulated categories equipped with a compact silting (or tilting) generator. Our first main result summarizes the relationships among the finitistic dimensions of $\mathcal{T}$ arising from different perspectives.

 \vspace{2mm} \noindent{\bf Theorem A}\label{Th1.4} {\rm(Theorems \ref{lem4.2} and \ref{lem4.13}).} {\it{Let $\mathcal{T}$ be a compactly generated triangulated category with compact silting generator
$M$. Then
 \begin{center}$\operatorname{fin.dim} \mathcal{T}^c\leq\mathrm{fpd}(\mathcal{T},M)
\leq\mathrm{findim}(\mathcal{T}^c,M)
\leq
 \mathrm{max}\{0,\mathrm{fpd}(\mathcal{T},M)\}\leq\mathrm{fpd}(\mathcal{T},M)-\mathrm{inf}M$.\end{center}In additional, if $M$ is tilting, then $\mathrm{fpd}(\mathcal{T},M)=\mathrm{fpd}(\mathcal{H}_M)=\mathrm{findim}(\mathcal{T}^c,M)$ and
$\mathrm{fpd}(\mathcal{T},M)<\infty\Longleftrightarrow\operatorname{fin.dim} \mathcal{T}^c<\infty$,  where $\mathrm{fpd}(\mathcal{H}_M)
=\mathrm{sup}\{\mathrm{projdim}_{\mathcal{H}_M}H\hspace{0.03cm}|\hspace{0.03cm}
H\in\mathcal{H}_M\ \textrm{has\ a\ resolution}\ 0\rightarrow P_n\rightarrow\cdots\rightarrow P_0\rightarrow M\rightarrow 0\ \textrm{in}\ \mathcal{H}_M\ \textrm{with}\ P_i\in\mathrm{add}(M)\}$.}}

 \vspace{2mm}
The theorem shows that these distinct notions of finitistic dimension are mutually bounded and coincide under the tilting condition. This offers a flexible framework for translating between different categorical settings. Moreover, it reframes the finitistic dimension conjecture: rather than relying on the detailed structure of a ring, this perspective is based on intrinsic properties of triangulated categories, such as generativity and $t$-structures. In the context of tilting theory,  the theorem implies that finiteness of the finitistic dimension is preserved under derived equivalence. Thus, these results unify existing notions of finitistic dimension, lift the classical conjecture to the triangulated level, and reduces the problem to the abelian heart $\mathcal{H}_M$ in the tilting case. We also obtained analogous results for various (big) finitistic dimensions (Theorem \ref{lem3.2}).

A natural extension of this categorical approach is the study of recollements, introduced by Beilinson-Bernstein-Deligne \cite{BBD} to decompose derived categories of sheaves into open and closed parts. Recollements of derived module categories can be viewed as short exact sequences that describe a derived module category in terms of a subcategory and a quotient, providing a useful framework for understanding connections among three algebraic or geometric objects. Our second main result builds on this framework by providing explicit inequalities for global dimension and (big) finitistic dimension under recollements of triangulated categories, thereby significantly generalizing classical ring-theoretic results on homological dimensions in the recollement setting.

\vspace{2mm} \noindent{\bf Theorem B}\label{Th1.4} {\rm(Theorems \ref{lem6.1}, \ref{lem7.1} and \ref{lem6.2}).} {\it{Let $\mathcal{R},\mathcal{S},\mathcal{T}$ be compactly generated triangulated categories with compact silting generators
$M_\mathcal{R},M_\mathcal{S},M_\mathcal{T}$ forming a recollement:
\begin{center}$\xymatrix@C=25pt@R=8pt{
\mathcal{ R}\ar[rr]|{i_\ast=i_!} &&\mathcal{S}\ar@<-1.5ex>[ll]_{i^\ast}\ar@<+1.5ex>[ll]^{i^!} \ar[rr]|{{j^!=j^\ast}} & &\mathcal{T}\ar@<-1.5ex>[ll]_{j_!} \ar@<+1.5ex>[ll]^{j_\ast} }$\end{center}Under the hypotheses that $\mathrm{pdim}_M(i_{*}(M_{\mathcal{R}}))<\infty$, we establish the following inequalities
\begin{itemize}
\item \textbf{(Big finitistic dimension)}
\[
\mathrm{FPD}(\mathcal{S}, M_{\mathcal{S}}) \leq \mathrm{FPD}(\mathcal{R}, M_{\mathcal{R}}) + \mathrm{FPD}(\mathcal{T}, M_{\mathcal{T}}) + C,
\]
where \( C=-\mathrm{inf}M_\mathcal{S}
+\operatorname{projdim}_{\mathcal{S}}i_*(M_\mathcal{R})+\mathrm{sup}i_\ast(M_\mathcal{R})
+\mathrm{projdim}_\mathcal{S}j_!(M_\mathcal{T})+\mathrm{sup}j_!(M_\mathcal{T})
+1 \). Moreover, \( \mathrm{FPD}(\mathcal{R}, M_{\mathcal{R}}) \) and \( \mathrm{FPD}(\mathcal{T}, M_{\mathcal{T}}) \) are each bounded by \( \mathrm{FPD}(\mathcal{S}, M_{\mathcal{S}}) \) plus explicit constants.

\item \textbf{(Finitistic dimension)} For the finitistic dimension \( \mathrm{fpd} \) defined via compact objects, similar bounds are obtained under the stronger assumption \( i_{*}(M_{\mathcal{R}}) \in \mathcal{S}^{c} \).

 \item \textbf{(Global dimension)} Analogous inequalities hold for \( \mathrm{gl.dim} \), leading to a characterization: \( \mathrm{gl.dim}\mathcal{S} < \infty\Longleftrightarrow \mathrm{gl.dim}\mathcal{R} < \infty \) and \( \mathrm{gl.dim}\mathcal{T} < \infty \).
\end{itemize}}}

\vspace{2mm}
These results unify and extend the classical recollement inequalities for rings due to Chen and Xi \cite{CX}. As applications, we recover and refine known bounds for homological dimensions of triangular matrix rings, idempotent recollements, exact contexts, and trivial extensions within this unified categorical framework.

Structure of the paper. Section 2 recalls necessary definitions and facts about compactly generated triangulated categories. Section 3 introduces projective and injective dimensions for objects in such categories, providing several equivalent characterizations via resolutions and the vanishing of functors.  In Section 4, we study various (big) finitistic dimensions, prove their relationships (Theorem A), and show that for a tilting generator $M$,
$\mathrm{FPD}(\mathcal{T},M)=\mathrm{FPD}(\mathcal{H}_M)$ (see Theorem \ref{lem3.2}). Section 5 contains our main results on how these dimensions behave under recollements (Theorem B), together with applications to (dg) rings.

\bigskip
\section{\bf Preliminaries}
This section is devoted to drawing some basic consequences
for use throughout this paper.  For terminology we shall follow \cite{BBD,CCZ,HKM,N0,Y1}.

 Let $\mathcal{T}$ be a triangulated category with the shift functor $[1]$. Given two full subcategories $\mathcal{X},\mathcal{Y}\subseteq\mathcal{T}$, denote by $\mathcal{X}\ast\mathcal{Y}$ the collection of objects $Z\in\mathcal{T}$ appearing
in an exact triangle $X\rightarrow Z\rightarrow Y\rightarrow X[1]$ with $X\in\mathcal{X}$ and $Y\in\mathcal{Y}$. For an interval $\Theta$ of integers, denote \begin{center}$\mathcal{X}^{\perp_\Theta}:=\{Y\in\mathcal{T}\hspace{0.02cm}|\hspace{0.02cm}\mathrm{Hom}_\mathcal{T}(X,Y[i])=0\ \textrm{for}\ X\in\mathcal{X}\ \textrm{and}\ i\in \Theta\},$\end{center} \begin{center}$^{\perp_\Theta}\mathcal{X}:=\{Y\in\mathcal{T}\hspace{0.02cm}|\hspace{0.02cm}\mathrm{Hom}_\mathcal{T}(Y[i],X)=0\ \textrm{for}\ X\in\mathcal{X}\ \textrm{and}\ i\in \Theta\},$\end{center}where $\Theta$ sometimes is represented by symbols $> 0$ or $\leq0$ and so on with the obvious meaning. The symbol $\perp$ stands for $\perp_0$.

\begin{df} {\rm (\cite{BBD})
Let  \( \mathcal{A} \subseteq \mathcal{T} \) be a subcategory and \( G \in \mathcal{T} \) an object.

(1) \(\operatorname{smd}(\mathcal{A})\) (resp. \(\operatorname{add}(\mathcal{A})\), \(\operatorname{Add}(\mathcal{A})\), \(\operatorname{Prod}(\mathcal{A})\) denotes the full subcategory of \( \mathcal{T} \) consisting of all direct summands (resp. finite direct sums, coproducts, products) of objects in \( \mathcal{A} \).

(2) For \( n > 0 \), define inductively:
\[
\operatorname{coprod}_1(\mathcal{A}) := \operatorname{add}(\mathcal{A}),\
\operatorname{coprod}_{n+1}(\mathcal{A}) := \operatorname{coprod}_1(\mathcal{A}) * \operatorname{coprod}_n(\mathcal{A}),
\]
\[
\operatorname{Coprod}_1(\mathcal{A}) := \operatorname{Add}(\mathcal{A}),\
\operatorname{Coprod}_{n+1}(\mathcal{A}) := \operatorname{Coprod}_1(\mathcal{A}) * \operatorname{Coprod}_n(\mathcal{A}).
\]

(3) Set \(\operatorname{coprod}(\mathcal{A}) := \bigcup_{n>0} \operatorname{coprod}_n(\mathcal{A})\), and \(\operatorname{Coprod}(\mathcal{A})\) denotes the smallest full subcategory of \( \mathcal{T} \) containing \( \mathcal{A} \) and closed under taking coproducts and extensions.

  (4) For integers \( s \leq t \), let
$G[s,t] := \{ G[-i]\hspace{0.03cm}|\hspace{0.03cm} \textrm{for}\ s \leq i \leq t \}$.
We extend this notation to allow \( s \) or \( t \) to be infinite; for example,
$G(-\infty, t] := \{ G[-i]\hspace{0.03cm}|\hspace{0.03cm} \textrm{for}\ i \leq t \}$.
For \( n > 0 \), let
\[
\langle G \rangle_n^{[s,t]} := \operatorname{smd}(\operatorname{coprod}_n(G[s, t])),\
\langle G \rangle^{[s,t]} := \bigcup_{n>0} \langle G \rangle_n^{[s,t]},\ \langle G \rangle_n := \langle G \rangle_n^{(-\infty,\infty)},\ \langle G \rangle:=\bigcup_{n>0}\langle G \rangle_n.
\]

\item Let \( s \leq t \) be integers (possibly infinite) and  $n>0$. Set
\[
\overline{\langle G \rangle}_n^{[s,t]} := \operatorname{smd}(\operatorname{Coprod}_n(G[s,t])),\
\overline{\langle G \rangle}^{[s,t]} := \operatorname{smd}(\operatorname{Coprod}(G[s,t]))
\]}
\end{df}

{\bf t-structures.} An object $C$ in a triangulated category $\mathcal{T}$ is called \emph{compact} if $\mathrm{Hom}_\mathcal{T}(C,-)$ commutes with all existing coproducts. Compact objects form a thick subcategory
of $\mathcal{T}$, denoted by $\mathcal{T}^c$. If $\mathcal{T}$ is cocomplete, $\mathcal{T}^c$ is skeletally small and $(\mathcal{T}^c)^\bot=0$, then we say that $\mathcal{T}$ is \emph{compactly generated}. In this case,
for $C\in\mathcal{T}^c$, there is a unique object
of $\mathcal{T}$, up to isomorphism, representing the functor
$$\mathrm{Hom}_\mathbb{Z}(\mathrm{Hom}_\mathcal{T}(C,-),\mathbb{Q}/\mathbb{Z}):\mathcal{T}^{\mathrm{op}}\rightarrow\mathrm{Ab},$$
denote by $C^\ast$, it is called the \emph{Brown-Comenetz dual} of $C$.

A pair $(\mathcal{U},\mathcal{V})$ of subcategories  in a triangulated category $\mathcal{T}$ is said to be a \emph{torsion pair} if $\mathcal{U}^\bot=\mathcal{V}, {^\bot}\mathcal{V}=\mathcal{U}$ and
$\mathcal{U}\ast\mathcal{V}=\mathcal{T}$. Such a torsion pair is said to be a \emph{t-structure} if $\mathcal{U}[1]\subseteq\mathcal{U}$. If $(\mathcal{U},\mathcal{V})$ is a t-structure, then $\mathcal{U}\cap\mathcal{V}[1]$ is an abelian category, called the \emph{heart} of $(\mathcal{U},\mathcal{V})$.
Let $\mathcal{M}$ be an additively closed subcategory of  $\mathcal{T}$, i.e. $\mathrm{Add}(\mathcal{M})=
\mathcal{M}$. We say that $\mathcal{M}$ is \emph{silting} if $(\mathcal{M}^{\bot_{>0}}, \mathcal{M}^{\bot_{\leq0}})$ is a $t$-structure in $\mathcal{T}$. An object $M$ in $\mathcal{T}$ is \emph{silting} if $\mathrm{Add}(M)$ is a silting subcategory. It follows
from \cite[Theorem 2.1]{HKM} that the pair $(M^{\bot_{>0}},M^{\bot_{\leq0}})$ is a t-structure of $\mathcal{T}$ whose  abelian heart is $\mathcal{H}_M:=M^{\bot_{\neq0}}$.
If $\mathcal{M}$ is a silting subcategory with $\mathrm{Hom}_\mathcal{T}(\mathcal{M},\mathcal{M}[\neq0])=0$, then  $\mathcal{M}$ is
\emph{tilting}.

Let $\mathcal{T}$ be a compactly generated triangulated category and $M$ a compact silting generator with abelian heart  $\mathcal{H}_M$. Then $\mathcal{T}$ admits set-indexed coproducts and products by the
Brown representability theorem. Denote $\mathcal{P}=\mathrm{Add}(M)$ and $\mathcal{I}=\mathrm{Prod}(M^\ast)$. For $n\geq0$, put $$\mathcal{T}^{< n+1}=\mathcal{T}^{\leq n}:=\mathcal{P}^{\perp_{>n}},\qquad  \mathcal{T}^{\geq n+1}=\mathcal{T}^{>n}:=\mathcal{P}^{\perp_{\leq n}}$$ and set $\mathcal{T}^{-}:=\bigcup_{n>0}\mathcal{T}^{\leq n},\ \mathcal{T}^{+}:=\bigcup_{n>0}\mathcal{T}^{\geq -n},\ \mathcal{T}^\mathrm{b}:=\mathcal{T}^-\cap\mathcal{T}^+$.
 Let
$\mathcal{H}^0_M(-)$ be the cohomological functor $\mathcal{T}\rightarrow\mathcal{H}_M$ and $\mathcal{H}^{-\ell}_M(-):=\mathcal{H}^0_M(-[\ell])$. For any object
$X\in\mathcal{T}$ and $\ell\in\mathbb{Z}$, the truncation functors
 $\tau^{\leq \ell}$ and $\tau^{\geq \ell}$ yield canonical exact triangles
$$\tau^{\leq \ell}(X)\rightarrow X\rightarrow \tau^{\geq \ell+1}(X)\rightarrow\tau^{\leq \ell}(X)[1],$$
$$\tau^{\leq \ell-1}(X)\rightarrow \tau^{\leq \ell}(X)\rightarrow\mathcal{H}^{\ell}_M(X)[-\ell]\rightarrow\tau^{\leq \ell-1}(X)[1],$$
$$\tau^{\geq\ell+1}(X)[-1]\rightarrow\mathcal{H}^{\ell}_M(X)[-\ell]\rightarrow\tau^{\geq\ell}(X)\rightarrow \tau^{\geq\ell+1}(X).$$
We also set $\mathrm{sup}X:=\mathrm{sup}\{\ell\in\mathbb{Z}\hspace{0.03cm}|\hspace{0.03cm}\mathcal{H}^\ell_M(X)\neq0\}$, $\mathrm{inf}X:=\mathrm{inf}\{\ell\in\mathbb{Z}\hspace{0.03cm}|\hspace{0.03cm}\mathcal{H}^\ell_M(X)\neq0\}$.

\begin{lem}\label{lem0.2} {\rm (\cite[Lemma 2.5]{Y1})} For $P\in\mathcal{P}$ and $I\in\mathcal{I}$, $X\in\mathcal{T}$ and $\ell\in\mathbb{Z}$, one has $$\mathrm{Hom}_\mathcal{T}(P[-\ell],X)\cong\mathrm{Hom}_{\mathcal{T}}(P,\mathcal{H}^\ell_M(X))\cong\mathrm{Hom}_{\mathcal{H}_M}(\mathcal{H}^0_M(P),\mathcal{H}^\ell_M(X)),$$
$$\mathrm{Hom}_\mathcal{T}(X,I[-\ell])\cong\mathrm{Hom}_{\mathcal{T}}(\mathcal{H}^\ell_M(X),I)\cong\mathrm{Hom}_{\mathcal{H}_M}(\mathcal{H}^\ell_M(X),\mathcal{H}^0_M(I)).$$
\end{lem}

Let $X\in\mathcal{T}^{-}$. By Lemma \ref{lem0.2}, there is $f:P\rightarrow X$ with $P\in\mathcal{P}[-\mathrm{sup}X]$, such that $\mathcal{H}^{\mathrm{sup}X}_M(f)$ is surjective. We call $f$ a \emph{sppj morphism}.  A \emph{sppj resolution} $P_\bullet$ of $X$ is a sequence of exact triangles
\begin{align}
X_{n+1}\stackrel{g_{n+1}}\longrightarrow P_n\stackrel{f_n}\longrightarrow X_n\longrightarrow X_{n+1}[1]
\label{exact03}
\tag{$\dag_n$}\end{align}
such that $f_n$ is a sppj morphism for $n\geq 0$ with $X_0=X$. Then $\mathrm{sup}X_{n+1}\leq \mathrm{sup}X_n$ for all $n\geq0$.  Set $d_n=g_{n}f_n:P_n\rightarrow P_{n-1}$, and write $$\cdots\longrightarrow P_n\stackrel{d_n}\longrightarrow P_{n-1}\longrightarrow\cdots\longrightarrow P_1\stackrel{d_1}\longrightarrow P_0\longrightarrow X.$$
Similarly,  for $Y\in\mathcal{T}^{+}$, there exists $g:Y\rightarrow I$ with $I\in\mathcal{I}[-\mathrm{inf}Y]$ such that $\mathcal{H}^{\mathrm{inf}Y}_M(g)$ is injective.  We call $g$ an \emph{ifij morphism}.
 An \emph{ifij coresolution} $I^\bullet$ of $Y$ is a sequence of exact triangles
\begin{align}
Y^{n}\stackrel{f^{n}}\longrightarrow I^n\stackrel{g^{n+1}}\longrightarrow Y^{n+1}\longrightarrow Y^{n}[1]
\label{exact03}
\tag{$\dag^n$}\end{align}
such that $f^n$ is an ifij morphism for $n\geq 0$ with $Y^0=Y$. Then $\mathrm{inf}Y^{n+1}\geq \mathrm{inf}Y^n$  for all $n\geq0$.  Set $d^n=f^{n+1}g^{n+1}:I^n\rightarrow I^{n+1}$, and write $$Y\longrightarrow I^0\stackrel{d^0}\longrightarrow I^{1}\longrightarrow\cdots\longrightarrow I^n\stackrel{d^n}\longrightarrow I^{n+1}\longrightarrow \cdots.$$

\begin{rem}\label{lem:0.1} {\rm (1) Let $X\in\mathcal{T}^{-}$ and $f:P'\rightarrow X$ be a sppj morphism. One has an exact triangle $Z'\stackrel{b}\rightarrow P'\stackrel{a}\rightarrow X\rightarrow Y'[1]$. So the following triangle $$Z'\oplus P'\stackrel{\left[\begin{smallmatrix}b &0\\ 0 &1\end{smallmatrix}\right]}\longrightarrow P'\oplus P'\stackrel{[a,0]}\longrightarrow X\longrightarrow (Y'\oplus P')[1]$$is exact. Set $P=P'\oplus P'$ and $Z=Z'\oplus P'$. Then $\mathrm{sup}Z=\mathrm{sup}X$ and $f=[a,0]$ is a sppj morphism, that is,
there is an exact triangle $Z\rightarrow P\stackrel{f}\rightarrow X\rightarrow Y[1]$, such that $g$ is a sppj morphism and $\mathrm{sup}Y=\mathrm{sup}X$. Similarly, for $Y\in\mathcal{T}^{+}$, there exists an exact triangle $Y\stackrel{g}\rightarrow I\rightarrow W\rightarrow Y[1]$, such that $g$ is an ifij morphism and $\mathrm{inf}Y=\mathrm{inf}W$.

(2) It follows from Lemma \ref{lem0.2} that $\mathrm{inf}M^\ast=0$ and $\mathrm{sup}M^\ast=-\mathrm{inf}M$. In fact, $M^\ast$ is cosilting, i.e., $(^{\bot_{\geq0}}M^\ast,^{\bot_{<0}}M^\ast)$ is a $t$-structure.}
\end{rem}

\bigskip
\section {\bf Projective and injective dimensions for objects in $\mathcal{T}$}
We work in a fixed compactly generated triangulated category $\mathcal{T}$ with a compact silting generator $M$ and associated heart $\mathcal{H}_M$. This section introduces projective and injective dimensions for objects in $\mathcal{T}$ and provides various criteria for recognizing them, for the convenience of later discussion.

\begin{df}\label{lem:2.2}{\rm (1) The \emph{projective dimension} $\mathrm{projdim}_\mathcal{T}X$ of an object $X\in\mathcal{T}$ is defined as $$\mathrm{projdim}_\mathcal{T}X=\mathrm{inf}\{n\in\mathbb{Z}\hspace{0.03cm}|\hspace{0.03cm}
\mathrm{Hom}_\mathcal{T}(X,U[i])=0\ \textrm{for\ all}\ U\in\mathcal{T}^{-}\ \textrm{and}\ i>n+\mathrm{sup}U\}.$$

(2) The \emph{injective dimension}  $\mathrm{injdim}_\mathcal{T}Y$ of an object $Y\in\mathcal{T}$  is defined as $$\mathrm{injdim}_\mathcal{T}Y=\mathrm{sup}\{n\in\mathbb{Z}\hspace{0.03cm}|\hspace{0.03cm}
\mathrm{Hom}_\mathcal{T}(V[i],Y)=0\ \textrm{for\ all}\ V\in\mathcal{T}^{+}\ \textrm{and}\ i<-n+\mathrm{inf}V\}.$$}
\end{df}

We now establish some useful criteria for the projective dimensions of objects.

\begin{prop}\label{lem2.3} Let $X\in\mathcal{T}$ and $n\in\mathbb{Z}$. The following are equivalent:

$(1)$ $\mathrm{projdim}_\mathcal{T}X\leq n$;

$(2)$ $\mathrm{Hom}_\mathcal{T}(X,U[i])=0$ for all $U\in\mathcal{T}^{\mathrm{b}}$ and $i>n+\mathrm{sup}U$;

$(3)$ $\mathrm{Hom}_\mathcal{T}(X,\mathcal{H}_M[i])=0$ for all $i>n$.\\
In particular, $\mathrm{projdim}_\mathcal{T}X=\mathrm{sup}\{n\in\mathbb{Z}\hspace{0.03cm}|\hspace{0.03cm}
\mathrm{Hom}_\mathcal{T}(X,\mathcal{H}_M[n])\neq0\}$. If $\mathrm{projdim}_\mathcal{T}X<\infty$, then
$\mathrm{projdim}_\mathcal{T}X=\mathrm{sup}\{n\in\mathbb{Z}\hspace{0.03cm}|\hspace{0.03cm}
\mathrm{Hom}_\mathcal{T}(X,\mathcal{P}[n])\neq0\}$.
\end{prop}
\begin{proof} (1) $\Rightarrow$ (2) $\Rightarrow$ (3) are trivial.
(3) $\Rightarrow$ (2) By the induction on $\mathrm{sup}U-\mathrm{inf}U$.

(2) $\Rightarrow$ (1) By analogy with the proof of (3) $\Rightarrow$ (2) in \cite[Theorem 3.3]{Y1}.

 If $\mathrm{projdim}_\mathcal{T}X=n$, then $\mathrm{Hom}_\mathcal{T}(X,H[n])\neq0$ for some $H\in\mathcal{H}_M$. Consider the exact triangle $H'\rightarrow P\stackrel{f}\rightarrow H\rightarrow H'[1]$ with $f$ a sppj morphism. As $\mathrm{sup}H'\leq0$, $\mathrm{Hom}_\mathcal{T}(X,H'[n+1])=0$, it follows from the exact sequence $\mathrm{Hom}_\mathcal{T}(X,P[n])\rightarrow\mathrm{Hom}_\mathcal{T}(X,H[n])\rightarrow0$ that $\mathrm{Hom}_\mathcal{T}(X,P[n])\neq0$. Conversely, if $\mathrm{projdim}_\mathcal{T}X=m> n$, then $\mathrm{Hom}_\mathcal{T}(X,\mathcal{H}_M[m])\neq0$, so
  $\mathrm{Hom}_\mathcal{T}(X,\mathcal{P}[m])\neq0$ by the preceding proof, a contradiction, it means that $\mathrm{projdim}_\mathcal{T}X\leq n$. Also $\mathrm{Hom}_\mathcal{T}(X,P[n])\neq0$ for some $P\in\mathcal{P}$, one has
 $\mathrm{projdim}_\mathcal{T}X=n$.
\end{proof}

\begin{rem}\label{lem:2.1}{\rm  (1) Let $X\in\mathcal{T}$ and set $s=\mathrm{sup}X$. If $s=\infty$, then $\mathrm{projdim}_\mathcal{T}X\geq-s$. Assume that $s<\infty$.
Consider the following exact triangle $$\tau^{\geq1}(M^\ast)[-1]\rightarrow\mathcal{H}^{0}_M(M^\ast)\rightarrow M^\ast\rightarrow \tau^{\geq1}(M^\ast).$$ As $X[s]\in\mathcal{T}^{\leq0}$ and $M^\ast\in\mathcal{T}^{\geq0}$,
$\mathrm{Hom}_\mathcal{T}(X[s],\tau^{\geq1}(M^\ast)[-1])=0
=\mathrm{Hom}_\mathcal{T}(X[s],\tau^{\geq1}(M^\ast))$, so
$\mathrm{Hom}_\mathcal{T}(X[s],\mathcal{H}^{0}_M(M^\ast))\cong
\mathrm{Hom}_\mathcal{T}(X[s],M^\ast)
\cong\mathrm{Hom}_\mathcal{T}(\mathcal{H}^{s}_M(X),M^\ast)\neq0$ by \cite[Theorem 1.3]{HKM} and Lemma \ref{lem0.2}. Hence $\mathrm{projdim}_\mathcal{T}X\geq-\mathrm{sup}X$ by Proposition \ref{lem2.3}.

(2) Let $\mathrm{projdim}_\mathcal{T}X=n$. If $\mathrm{inf}M=-\infty$, then $\mathrm{inf}X\geq\mathrm{inf}M-\mathrm{projdim}_\mathcal{T}X$.
If $M\in\mathcal{T}^\mathrm{b}$, then $M^\ast\in\mathcal{T}^\mathrm{b}$. So $\mathrm{Hom}_\mathbb{Z}(\mathrm{Hom}_\mathcal{T}(M,X[-i]),\mathbb{Q}/\mathbb{Z})
 \cong\mathrm{Hom}_\mathcal{T}(X,M^\ast[i])=0$ for $i>n+\mathrm{sup}M^\ast$, it follows from Lemma \ref{lem0.2} that $\mathrm{Hom}_\mathcal{T}(M,\mathcal{H}^{-i}_M(X))\cong
 \mathrm{Hom}_\mathcal{T}(M[i],X)=0$ for $i>n-\mathrm{inf}M$. Hence $\mathrm{inf}X\geq\mathrm{inf}M-\mathrm{projdim}_\mathcal{T}X$.

 (3) Let $X\in\mathcal{T}$ and $n\in\mathbb{Z}$. Then
 $\mathrm{projdim}_\mathcal{T}X[n]=\mathrm{projdim}_\mathcal{T}X+n$. By Proposition \ref{lem2.3},
 our $\mathrm{projdim}_\mathcal{T}X$ coincides with the projective dimension with respect to $M$ (denoted $\mathrm{pdim}_M(X)$ in \cite{BCRPZ}) and also with the projective dimension relative to $\mathcal{H}_M$ (denoted
$\mathrm{pdim}_{\mathcal{H}_M}X$ in \cite{Ko}).

(4) Let $X\in\mathcal{T}^-$. Then $\mathrm{projdim}_\mathcal{T}X+\mathrm{sup}X$  coincides with the projective dimension $\mathrm{pd}_\mathcal{T}X$ defined in \cite[Definition 3.1]{Y1}. Let $X'\rightarrow P\stackrel{f}\rightarrow X\rightarrow X'[1]$ be an exact triangle with $f$ a sppj morphism. If $\mathrm{projdim}_\mathcal{T}X+\mathrm{sup}X\geq1$, then $\mathrm{projdim}_\mathcal{T}X'=\mathrm{projdim}_\mathcal{T}X-1$ by \cite[Proposition 3.6]{Y1}.}
\end{rem}

The next result characterizes the projective dimension through sppj resolution, generalizing \cite[Theorem 2.22]{Mi18} and \cite[Proposition 2.15]{Ko}.

\begin{prop}\label{lem2.5} Let $X\in\mathcal{T}^{-}$ and $d\geq0$. The following are equivalent:

$(1)$ $\mathrm{projdim}_\mathcal{T}X+\mathrm{sup}X=d$.

$(2)$ For any sppj resolution $P_\bullet$ of $X$, there exists $e\geq0$ such that $P_n=0$ for $n>e$ and $e+\mathrm{sup}P_0-\mathrm{sup}P_e=d$ and the triangle $X_{e}\rightarrow P_{e-1}\rightarrow X_{e-1}\rightarrow X_{e}[1]$ is not split;

$(3)$ $X$ has a sppj resolution of length $d$, such that $\mathrm{sup}P_i=\mathrm{sup}X$ for $0\leq i\leq d$ and the triangle $X_{d}\rightarrow P_{d-1}\rightarrow X_{d-1}\rightarrow X_{d}[1]$ is not split;

$(4)$ $d$ is the smallest number which satisfies $X\in\mathcal{P}[-\mathrm{sup}X]\ast\cdots\ast\mathcal{P}[-\mathrm{sup}X+d]$.
\end{prop}
\begin{proof} Set $s=\mathrm{sup}X$ and fix  a sppj resolution of $X$
$$\cdots\longrightarrow P_n\stackrel{d_n}\longrightarrow P_{n-1}\longrightarrow\cdots\longrightarrow P_1\stackrel{d_1}\longrightarrow P_0\longrightarrow X.$$

 (1) $\Rightarrow$ (2) Consider the exact triangle $X_1\rightarrow P_0\rightarrow X\rightarrow X_1[1]$. We proceed by induction on $d$.
If $d=0$, then $\mathrm{Hom}_\mathcal{T}(X,X_1[1])=0$ as $\mathrm{sup}X_1\leq s$, so this triangle is split and $X\in\mathcal{P}[-s]$.
Assume that $d\geq1$. By Remark \ref{lem:2.1}(4), $\mathrm{projdim}_\mathcal{T}X_1+\mathrm{sup}X_1=d+\mathrm{sup}P_1-\mathrm{sup}P_0-1\leq d-1$. As $\cdots\rightarrow P_n\rightarrow\cdots\rightarrow P_1\rightarrow X_1$ is a sppj resolution of $X_1$, it follows from the induction that
there is $e\geq0$ such that $P_n=0$ for $n>e$ and $e-1+\mathrm{sup}P_1-\mathrm{sup}P_e=d+\mathrm{sup}P_1-\mathrm{sup}P_0-1$ and the triangle $X_{e}\rightarrow P_{e-1}\rightarrow X_{e-1}\rightarrow X_{e}[1]$ is not split, as claimed.

(2) $\Rightarrow$ (3) By Remark \ref{lem:0.1}(1), one can choose the sppj resolution $P^\bullet$, such that $\mathrm{sup}P_i=\mathrm{sup}X$ for all $i\geq 0$. In this case, $e=d$, as desired.

(3) $\Rightarrow$ (4) Set $X_0=X$. The corresponding exact triangles$$X_{n+1}\rightarrow P_n\rightarrow X_n\rightarrow X_{n+1}[1]$$
imply that $X\in\mathcal{P}[-s]\ast
\mathcal{P}[-s+1]\ast\cdots
\ast\mathcal{P}[-s+d]$ as $\mathrm{sup}X_{n}=\mathrm{sup}X$ for $n\geq0$. If $X\in\mathcal{P}[-s]\ast\cdots
\ast\mathcal{P}[-s+d-1]$, then $\mathrm{Hom}_\mathcal{T}(X,X_d[i])=0$ for $i> d$
as $\mathrm{Hom}_\mathcal{T}(M,X_d[i])=0$ for $i>\mathrm{sup}X_d$.
Hence the corresponding triangles imply that $$\mathrm{Hom}_\mathcal{T}(X_{d-1},
X_d[1])\cong\mathrm{Hom}_\mathcal{T}(X,X_d[d])=0,$$ so the triangle $X_{d}\rightarrow P_{d-1}\rightarrow X_{d-1}\rightarrow X_{d}[1]$ is split, a contradiction.

(4) $\Rightarrow$ (1) We may assume that $d\geq1$. As $\mathrm{Hom}_\mathcal{T}(M,\mathcal{H}_M[\neq0])=0$ and $X\in\mathcal{P}[-s]\ast\cdots\ast\mathcal{P}[-s+d]$, it follows that $\mathrm{Hom}_\mathcal{T}(X,\mathcal{H}_M[i])=0$ for $i>-s+d$. Also
 there exists an exact triangle $Y\rightarrow Q\rightarrow X\rightarrow Y[1]$ such that $Q\in\mathcal{P}[-s]$ and $d-1$ is the smallest number which satisfies $Y\in\mathcal{P}[-s]\ast\cdots\ast\mathcal{P}[-s+d-1]$. By the induction, $\mathrm{Hom}_\mathcal{T}(X,\mathcal{H}_M[-s+d])
 \cong\mathrm{Hom}_\mathcal{T}(Y,\mathcal{H}_M[-s+d-1])\neq0$, so $\mathrm{projdim}_\mathcal{T}X+\mathrm{sup}X=d$
by Proposition \ref{lem2.3}.
\end{proof}

Dually, we obtain analogous characterizations for the injective dimension.

\begin{prop}\label{lem3.3} Let $Y\in\mathcal{T}$ and $n\in\mathbb{Z}$. The following are equivalent:

$(1)$ $\mathrm{injdim}_\mathcal{T}Y\leq n$;

$(2)$ $\mathrm{Hom}_\mathcal{T}(V[i],Y)=0$ for all $V\in\mathcal{T}^{\mathrm{b}}$ and $i<-n+\mathrm{inf}V$;

$(3)$ $\mathrm{Hom}_\mathcal{T}(\mathcal{H}_M[i],Y)=0$ for $i<-n$.\\
In particular, $\mathrm{injdim}_\mathcal{T}Y=\mathrm{sup}\{n\in\mathbb{Z}\hspace{0.03cm}|\hspace{0.03cm}
\mathrm{Hom}_\mathcal{T}(\mathcal{H}_M[-n],Y)\neq0\}$. If $\mathrm{injdim}_\mathcal{T}Y<\infty$, then $\mathrm{injdim}_\mathcal{T}Y=\mathrm{sup}\{n\in\mathbb{Z}\hspace{0.03cm}|\hspace{0.03cm}
\mathrm{Hom}_\mathcal{T}(\mathcal{I}[-n],Y)\neq0\}$.
\end{prop}

\begin{prop}\label{lem3.4}  Let $Y\in\mathcal{T}^{+}$ and $d\geq0$. The following are equivalent:

$(1)$ $\mathrm{injdim}_\mathcal{T}Y-\mathrm{inf}Y=d$;

$(2)$ For any inij resolution $I^\bullet$ of $Y$, there exists $e\geq0$ such that $I^n=0$ for $n>e$ and $e+\mathrm{inf}I^e-\mathrm{inf}I^0=d$ and the triangle $Y^{e-1}\rightarrow I^{e-1}\rightarrow Y^{e}\rightarrow Y^{e-1}[1]$ is not split;

$(3)$ $Y$ has an ifij resolution of length $d$, such that $\mathrm{inf}I^i=\mathrm{inf}Y$ for $0\leq i\leq d$ and the triangle $Y^{d-1}\rightarrow I^{d-1}\rightarrow Y^{d}\rightarrow Y^{d-1}[1]$ is not split;

$(4)$ $d$ is the smallest number which satisfies $Y\in\mathcal{I}[-\mathrm{inf}Y-d]\ast\cdots\ast\mathcal{I}[-\mathrm{inf}Y]$.
\end{prop}

Following \cite{CCZ,KPV}, set $\mathcal{T}^{\mathrm{sb}}:=\bigcup_{i\geq0}\mathrm{smd}(\mathrm{Coprod}(M[-i,i]))$, $\mathcal{T}^\mathrm{b}_p:=\mathcal{T}^-_p\cap\mathcal{T}^+_p$ and $\mathcal{T}^\mathrm{b}_i:=\mathcal{T}^-_i\cap\mathcal{T}^+_i$, where $\mathcal{T}^-_p:={^{\bot_\ll}}(\mathcal{T}^\mathrm{b})$, $\mathcal{T}^+_p:={^{\bot_\gg}}(\mathcal{T}^\mathrm{b})$,
 $\mathcal{T}^-_i:=(\mathcal{T}^\mathrm{b})^{\bot_\gg}$, $\mathcal{T}^+_i:=(\mathcal{T}^\mathrm{b})^{\bot_\ll}$.

\begin{rem}\label{lem:2.2}{\rm  (1) By Proposition \ref{lem2.3}, one has $\mathcal{T}^+_p=\{X\in\mathcal{T}\hspace{0.03cm}|\hspace{0.03cm}
 \mathrm{projdim}_\mathcal{T}X<\infty\}$, it follows from  \cite[Proposition 3.4]{CCZ} that $\mathcal{T}^{\mathrm{sb}}=\{X\in\mathcal{T}^-\hspace{0.03cm}|\hspace{0.03cm}
 \mathrm{projdim}_\mathcal{T}X<\infty\}\subseteq\mathcal{T}^+_p$.

  (2) Let $X\in\mathcal{T}^-$. Then $\mathrm{Hom}_\mathcal{T}(X,U[i])=0$ for $U\in\mathcal{T}^{\mathrm{b}}$ and $i<\mathrm{inf}U-\mathrm{sup}X$, so
 $\mathcal{T}^-\subseteq\mathcal{T}^-_p$. If $M\in\mathcal{T}^\mathrm{b}$, then $M^\ast\in\mathcal{T}^\mathrm{b}$, and so $\mathcal{T}^-=\mathcal{T}^-_p$ by Lemma \ref{lem0.2}. Thus
$\mathcal{T}^\mathrm{b}_p=
  \{X\in\mathcal{T}^\mathrm{b}\hspace{0.03cm}|\hspace{0.03cm}
 \mathrm{projdim}_\mathcal{T}X<\infty\}$ by Remark \ref{lem:2.1}(2), and $\mathcal{T}^{\mathrm{sb}}=\mathcal{T}^{\mathrm{b}}_p$ by  \cite[Proposition 3.4]{CCZ}.

(3) If $M\in\mathcal{T}^\mathrm{b}$, then $\mathcal{T}^+=\mathcal{T}^+_i$ by Lemma \ref{lem0.2}, and hence $\mathcal{T}^\mathrm{b}_i=
  \{Y\in\mathcal{T}^\mathrm{b}\hspace{0.03cm}|\hspace{0.03cm}
\mathrm{injdim}_\mathcal{T}Y<\infty\}$ as $\mathcal{T}^-_i=\{Y\in\mathcal{T}\hspace{0.03cm}|\hspace{0.03cm}
 \mathrm{injdim}_\mathcal{T}Y<\infty\}$. Let $Y\in\mathcal{T}$. Then $\mathrm{injdim}_\mathcal{T}Y\geq\mathrm{inf}Y$, $\mathrm{sup}Y\leq\mathrm{injdim}_\mathcal{T}Y-\mathrm{inf}M$,  and
 $\mathrm{injdim}_\mathcal{T}Y[n]=\mathrm{injdim}_\mathcal{T}Y-n$ for $n\in\mathbb{Z}$.

 (4) Following \cite{CCZ}, the \emph{global dimension} of $\mathcal{T}$ with respect to $M$ is defined as $$\mathrm{gpd}(\mathcal{T},M):=
 \mathrm{sup}\{\mathrm{projdim}_\mathcal{T}X\hspace{0.03cm}|\hspace{0.03cm}
X\in\mathcal{T}^{\mathrm{b}}\cap\mathcal{T}^{\geq0}\}.$$By \cite[Lemma 5.2]{CCZ}, $\mathrm{gpd}(\mathcal{T},M)=\mathrm{sup}\{\mathrm{projdim}_\mathcal{T}H\hspace{0.03
cm}|\hspace{0.03cm}H\in\mathcal{H}_M\}=\mathrm{gl.dim}_{\mathcal{H}_M}\mathcal{T}$, the \emph{global dimension of $\mathcal{T}$ relative to $\mathcal{H}_M$}. Thus by \cite[Theorem 2.9]{Ko}, $\mathrm{gpd}(\mathcal{T},M)<\infty$ if and only if $\mathcal{T}^{\mathrm{b}}\subseteq\mathcal{T}^{\mathrm{b}}_p$, and if $M\in\mathcal{T}^\mathrm{b}$ then $\mathrm{gpd}(\mathcal{T},M)<\infty$ if and only if $\mathcal{T}^{\mathrm{b}}=\mathcal{T}^{\mathrm{b}}_p$.

(5) If $R$ is a ring, then $\mathrm{projdim}_RX=\mathrm{projdim}_{\mathrm{D}(R)}X$ and $\mathrm{injdim}_RX=\mathrm{injdim}_{\mathrm{D}(R)}X$ for $X\in\mathrm{D}(R)$. Thus $\mathrm{gpd}(\mathrm{D}(R),R)=\mathrm{gldim}R$, the global dimension of $R$.

 (6) If $A$ is a non-positive dg ring, then $\mathrm{projdim}_{\mathrm{D}(A)}X=\mathrm{projdim}_AX$ and $\mathrm{injdim}_{\mathrm{D}(A)}X=\mathrm{injdim}_AX$ for any $X\in\mathrm{D}(A)$ by \cite[Definition 2.1]{BSSW}, it follows from \cite[Theorem 5.1]{Mi18} that $\mathrm{gpd}(\mathrm{D}(A),A)=\mathrm{gldim}A$, the global dimension of $A$.}
\end{rem}

\begin{prop}\label{lem1.1} The following statements  are equivalent:

 $(1)$ $\mathrm{gpd}(\mathcal{T},M)<\infty$;

 $(2)$ $\mathcal{T}^{b}\subseteq\mathcal{T}^{b}_p$;

  $(3)$ $\mathcal{T}^{b}\subseteq\mathcal{T}^{b}_i$;

 $(4)$ $\mathcal{T}^{+}\subseteq\mathcal{T}^{+}_p$;

  $(5)$ $\mathcal{T}^{-}\subseteq\mathcal{T}^{-}_i$.\\Moreover, if
 $M\in\mathcal{T}^{b}$ then all the above inclusions are actually equalities.
\end{prop}
\begin{proof} (1) $\Leftrightarrow$ (2) This follows from  \cite[Theorem 2.9]{Ko}.

(2) $\Leftrightarrow$ (3) Let $X\in\mathcal{T}^{b}$. Then $\mathrm{Hom}_\mathcal{T}(H,X[t])=0$ for all $H\in\mathcal{H}_M$ and $t>\mathrm{projdim}_\mathcal{T}H+\mathrm{sup}X$, it implies that $\mathrm{injdim}_\mathcal{T}X\leq\mathrm{projdim}_\mathcal{T}H+\mathrm{sup}X$. So $\mathcal{T}^{b}\subseteq\mathcal{T}^{b}_i$. The converse is similar.

(4) $\Rightarrow$ (2) Since $\mathcal{T}^{b}\subseteq\mathcal{T}^{+}\subseteq\mathcal{T}^{+}_p$ by (4), it follows that $\mathcal{T}^{b}\subseteq\mathcal{T}^{-}_p\cap\mathcal{T}^{+}_p=\mathcal{T}^{b}_p$.

(1) $\Rightarrow$ (4) Let $\mathrm{gpd}(\mathcal{T},M)=d<\infty$ and $X\in\mathcal{T}^{+}$.
There is a sequence of objects in  $\mathcal{T}^\mathrm{b}$
$$\tau^{\leq\mathrm{inf}X}(X)\stackrel{\nu_0}\longrightarrow\tau^{\leq\mathrm{inf}X+1}(X)
\stackrel{\nu_1}\longrightarrow\tau^{\leq\mathrm{inf}X+2}(X)\longrightarrow\cdots,$$ it yields an exact triangle $\coprod_{j\geq0}\tau^{\leq\mathrm{inf}X+j}(X)\rightarrow \coprod_{j\geq0}\tau^{\leq\mathrm{inf}X+j}(X)\rightarrow X\rightarrow
\coprod_{j\geq0}\tau^{\leq\mathrm{inf}X+j}(X)[1]$. As $\mathrm{Hom}_\mathcal{T}(\tau^{\leq\mathrm{inf}X+j}(X),\mathcal{H}_M[d-\mathrm{inf}X+t])=0$ for $t\geq1$ and $j\geq0$, it follows that $\mathrm{Hom}_\mathcal{T}(\coprod_{j\geq0}\tau^{\leq\mathrm{inf}X+j}(X),\mathcal{H}_M[d-\mathrm{inf}X+t])=0$ for $t\geq1$, so $\mathrm{Hom}_\mathcal{T}(X,\mathcal{H}_M[d-\mathrm{inf}X+t])=0$ for $t\geq2$. Thus $\mathrm{projdim}_\mathcal{T}X<\infty$, and then $\mathcal{T}^{+}\subseteq\mathcal{T}^{+}_p$.

(1) $\Rightarrow$ (5) $\Rightarrow$ (3) The proof is similar to that of (1) $\Rightarrow$ (4) $\Rightarrow$ (2).
\end{proof}

The following result is a generalization of \cite[Proposition 5.3]{Mi18}.

\begin{prop}\label{lem0.4} Let $\mathcal{R},\mathcal{S}$ be compactly generated triangulated categories with compact silting generators
$M_\mathcal{R},M_\mathcal{S}$. If $i:\mathcal{R}\rightarrow\mathcal{S}$ is a triangulated equivalence, then $$\mathrm{gpd}(\mathcal{R},M_\mathcal{R})\leq\mathrm{gpd}(\mathcal{S},M_\mathcal{S})+\mathrm{projdim}_\mathcal{S}i(M_\mathcal{R})
+\mathrm{sup}i(M_\mathcal{R}).$$Consequently, $\mathrm{gpd}(\mathcal{R},M_\mathcal{R})<\infty$ if and only if $\mathrm{gpd}(\mathcal{S},M_\mathcal{S})<\infty$.
\end{prop}
\begin{proof} Let $j:\mathcal{S}\rightarrow\mathcal{R}$ be the inverse of $i$. Clearly, $i(M_\mathcal{R})\in\mathcal{S}^c\subseteq\mathcal{S}^-$, it follows from Proposition \ref{lem2.5} that $\mathrm{projdim}_\mathcal{S}i(M_\mathcal{R})<\infty$ and $\mathrm{sup}i(M_\mathcal{R})<\infty$.
As $\mathrm{Hom}_\mathcal{R}(M_\mathcal{R},j(M_\mathcal{S})[t])\cong
\mathrm{Hom}_\mathcal{R}(i(M_\mathcal{R}),M_\mathcal{S}[t])$ for $t\in\mathbb{Z}$, $\mathrm{sup}j(M_\mathcal{S})\leq\mathrm{projdim}_\mathcal{S}i(M_\mathcal{R})$ by Lemma \ref{lem0.2} and Proposition \ref{lem2.3}.
Let $H\in\mathcal{H}_{M_\mathcal{R}}$. Since $\mathrm{Hom}_\mathcal{S}(M_\mathcal{S}[t],i(H))\cong
\mathrm{Hom}_\mathcal{R}(j(M_\mathcal{S})[t],H)\ \textrm{for}\ t>\mathrm{sup}j(M_\mathcal{S})$, one has $\mathrm{inf}i(H)\geq-\mathrm{projdim}_\mathcal{S}i(M_\mathcal{R})$.
For $P\in\mathrm{Add}(M_\mathcal{R})$ and $i\in\mathbb{Z}$, as $$\mathrm{Hom}_\mathcal{R}(H,P[t])\cong\mathrm{Hom}_\mathcal{S}(i(H),i(P)[t])\ \textrm{for}\ t>\mathrm{projdim}_\mathcal{S}i(H)+\mathrm{sup}i(P),$$it implies that $\mathrm{projdim}_\mathcal{R}H\leq\mathrm{projdim}_\mathcal{S}i(H)+\mathrm{sup}i(P)
\leq\mathrm{gpd}(\mathcal{S},M_\mathcal{S})+\mathrm{projdim}_\mathcal{S}i(M_\mathcal{R})
+\mathrm{sup}i(M_\mathcal{R})$, the desired inequality holds true.
\end{proof}

Let $\mathcal{A}$ be an abelian category whose derived category $\mathrm{D}(\mathcal{A})$ has $\operatorname{Hom}$ sets and arbitrary (small) coproducts. An object $T$ of $\mathcal{A}$ is called \emph{$n$-tilting} ($n\geq 0$) if
$T^{(\Lambda)}$ is the coproduct of $\Lambda$ copies of $T$ both in $\mathcal{A}$ and $\mathrm{D}(\mathcal{A})$ for each set $\Lambda$;
 $\operatorname{Ext}_{\mathcal{A}}^k(T,T^{(\Lambda)}) = 0$ for all $k > 0$ and sets $\Lambda$; $\mathrm{projdim}_\mathcal{A}T\leq n$;
there is a generating class $\mathcal{G}$ of $\mathcal{A}$ such that, for each $G \in \mathcal{G}$, there is an exact sequence $0 \to G \to T^0 \to \cdots \to T^n \to 0$
with each $T^k\in\operatorname{Add}(T)$. $T$ is called a \emph{classical tilting} object if it is a $n$-tilting object for some $n\geq0$ and is compact as an object of $\mathcal{D}(\mathcal{A})$. In this case, there is a triangulated equivalence
\[
i:\mathrm{D}(\mathcal{H}_T) \xrightarrow{\cong}\mathrm{D}(\mathcal{A})
\]
  where $\mathcal{H}_T$ is the heart of the associated $t$-structure $\tau_T = (T^{\perp>0}, T^{\perp<0})$ in $\mathrm{D}(\mathcal{A})$ by \cite[Proposition 2.3]{FMS};
$T$ is a compact generator of $\mathrm{D}(\mathcal{H}_T)$ and $i(T)=T$ is a compact silting object in $\mathrm{D}(\mathcal{A})$ by \cite[Corollary 2.5]{FMS} and \cite[Theoreem 3]{NSZ}.
Hence by Proposition \ref{lem0.4}, $$\mathrm{gpd}(\mathrm{D}(\mathcal{H}_T),T)\leq\mathrm{gpd}(\mathrm{D}(\mathcal{A}),T)
+\mathrm{projdim}_{\mathrm{D}(\mathcal{A})}T.$$ 

\bigskip
\section {\bf Finitistic dimensions of triangulated categories}
Throughout this section, $\mathcal{T}$, $M$ and $\mathcal{H}_M$ are as in the previous section. Our goal is to recall the various (big) finitistic dimensions of $\mathcal{T}$, to compare them, and to relate them to the (big) finitistic dimension of the heart $\mathcal{H}_M$.

 Following \cite{BCRPZ}, the \emph{big finitistic dimension} of $\mathcal{T}$ with respect to $M$ are defined as
$$\mathrm{FPD}(\mathcal{T},M):=\mathrm{sup}\{\mathrm{projdim}_\mathcal{T}X
\hspace{0.03cm}|\hspace{0.03cm}
X\in\mathcal{T}^\mathrm{b}\cap\mathcal{T}^{\geq0}\ \textrm{with}\ \mathrm{projdim}_\mathcal{T}X<\infty\}.$$ As $X[\mathrm{inf}X]\in\mathcal{T}^\mathrm{b}\cap\mathcal{T}^{\geq0}$ for $X\in\mathcal{T}^\mathrm{b}$, and $\mathrm{projdim}_\mathcal{T}X\leq\mathrm{projdim}_\mathcal{T}X+\mathrm{inf}X$ for $X\in\mathcal{T}^\mathrm{b}\cap\mathcal{T}^{\geq0}$, it follows that $\mathrm{FPD}(\mathcal{T},M)=\mathrm{sup}\{\mathrm{projdim}_\mathcal{T}X+\mathrm{inf}X
\hspace{0.03cm}|\hspace{0.03cm}
X\in\mathcal{T}^\mathrm{b}\ \textrm{with}\ \mathrm{projdim}_\mathcal{T}X<\infty\}$.
Following \cite{CCZ}, the \emph{big finitistic dimension} of $\mathcal{T}$ at $M$ is defined to be
$$\mathrm{Findim}(\mathcal{T},M):=\mathrm{inf}\{n\in\mathbb{N}\hspace{0.03cm}|\hspace{0.03cm}
\mathcal{T}^{\geq0}\cap\mathcal{T}^{\mathrm{sb}}\subseteq\bigcup_{m\geq0}\overline{\langle M\rangle}^{[0,m]}[n]\}.$$

The following theorem summarizes the close relationships among these big finitistic dimensions under suitable hypotheses.

\begin{thm}\label{lem3.2} $(1)$ $\mathrm{FPD}(\mathcal{T},M)\leq\mathrm{silp}\mathcal{T}\leq\mathrm{gpd}(\mathcal{T},M)$. If $\mathrm{gpd}(\mathcal{T},M)<\infty$, then $\mathrm{FPD}(\mathcal{T},M)=\mathrm{gpd}(\mathcal{T},M)$, where $\mathrm{silp}\mathcal{T}:=\mathrm{sup}\{\mathrm{injdim}_\mathcal{T}P
\hspace{0.03cm}|\hspace{0.03cm}P\in\mathcal{P}\}$.

$(2)$ $\mathrm{FPD}(\mathcal{T},M)
\leq\mathrm{Findim}(\mathcal{T},M)\leq
 \mathrm{max}\{0,\mathrm{FPD}(\mathcal{T},M)\}\leq\mathrm{FPD}(\mathcal{T},M)-\mathrm{inf}M$.

$(3)$ If $M$ is tilting, then $\mathrm{Findim}(\mathcal{T},M)=\mathrm{FPD}(\mathcal{T},M)=\mathrm{FPD}(\mathcal{H}_M)=\mathrm{sup}\{\mathrm{projdim}_{\mathcal{H}_M}H
\hspace{0.03cm}|\hspace{0.03cm}
H\in\mathcal{H}_M\ \textrm{with}\ \mathrm{projdim}_{\mathcal{H}_M}H<\infty\}$.
\end{thm}
\begin{proof} (1) Let $X\in\mathcal{T}^\mathrm{b}$ with $\mathrm{projdim}_\mathcal{T}X=n<\infty$. By Proposition \ref{lem2.3}, one has$$\mathrm{Hom}_\mathcal{T}(X[-n-\mathrm{inf}X+\mathrm{inf}X],P)
\cong\mathrm{Hom}_\mathcal{T}(X,P[n])\neq0$$ for some $P\in\mathcal{P}$, it implies that $\mathrm{injdim}_\mathcal{T}P\geq \mathrm{projdim}_\mathcal{T}X+\mathrm{inf}X$. So $\mathrm{FPD}(\mathcal{T},M)\leq\mathrm{silp}\mathcal{T}\leq\mathrm{gpd}(\mathcal{T},M)$, the equality is clear.

(2)  Let $\mathrm{Findim}(\mathcal{T},M)=d$ and $X\in\mathcal{T}^\mathrm{b}\cap\mathcal{T}^{\geq0}$ with $\mathrm{projdim}_\mathcal{T}X<\infty$. As $\mathcal{T}^{\mathrm{sb}}=\{X\in\mathcal{T}^-\hspace{0.03cm}|\hspace{0.03cm}
\mathrm{projdim}_\mathcal{T}X<\infty\}$,
 $X\in\bigcup_{m\geq0}\overline{\langle M\rangle}^{[0,m]}[d]$. For $p\geq0$, one has $\mathrm{Hom}_\mathcal{T}(M[-p],\mathcal{H}_M[t])=0$ for $t>0$, it follows that $\mathrm{Hom}_\mathcal{T}(X[-d],\mathcal{H}_M[t])=0$ for all $t>0$. Hence $\mathrm{projdim}_\mathcal{T}X\leq d$ by Proposition \ref{lem2.3}, it yields the first inequality.
Let $\mathrm{FPD}(\mathcal{T},M)=n$ and $X\in\mathcal{T}^\mathrm{sb}\cap\mathcal{T}^{\geq0}$. Then $X\in\mathcal{T}^\mathrm{b}$. Set $s=\mathrm{sup}X$ and $i=\mathrm{inf}X\geq0$. Then $-s\leq-s+i\leq\mathrm{projdim}_\mathcal{T}X+i\leq n$, so $0\leq\mathrm{projdim}_\mathcal{T}X+s\leq n+s$. By Proposition \ref{lem2.5},
$X\in\mathcal{P}[-s]\ast\cdots\ast\mathcal{P}[-s+n+s]$. If $n<0$, then
$X\in\mathcal{P}[-s]\ast\cdots\ast\mathcal{P}[n]\subseteq\bigcup_{m\geq0}\overline{\langle M\rangle}^{[0,m]}$.
If  $n\geq0$, then $X\in(\mathcal{P}[-s-n]\ast\cdots\ast\mathcal{P})[n]\subseteq\bigcup_{m\geq0}\overline{\langle M\rangle}^{[0,m]}[n]$. The second inequality holds true. By Remark \ref{lem:2.1}(2), $\mathrm{FPD}(\mathcal{T},M)\geq\mathrm{inf}M$, it yields the last inequality.

(3) Let $M$ be tilting and $H\in\mathcal{H}_M$. Then $\mathrm{prodim}_{\mathcal{H}_M}H= \mathrm{projdim}_\mathcal{T}H$ by \cite[Lemma 4.8]{Y1}, and so the inequality $\mathrm{FPD}(\mathcal{T},M)\geq\mathrm{FPD}(\mathcal{H}_M)$ is valid. Conversely,
let $X\in\mathcal{T}^\mathrm{b}$ with $\mathrm{projdim}_\mathcal{T}X=n<\infty$. We may assume that $\mathrm{sup}X=0$ and $i=\mathrm{inf}X\leq-1$. By Proposition \ref{lem2.5} and Remark \ref{lem:0.1}(1), one can choose a sppj resolution of $X$ $$P_n\rightarrow P_{n-1}\rightarrow\cdots\rightarrow P_0\rightarrow X$$  with the corresponding exact sequence $X_{s+1}\rightarrow P_s\rightarrow X_s\rightarrow X_{s+1}[1]$, such that $\mathrm{sup}X_s=0$ for $s\geq0$, $X_{n+1}=0$ and $X_{n-1}\not\in\mathcal{P}$, where $X_0=X$. For $s\geq0$, one has the following commutative diagram of exact triangles
\begin{center} $\xymatrix@C=25pt@R=20pt{
P_{s}[s-1] \ar@{=}[d]\ar[r]& X_{s}[s-1]\ar[d]\ar[r] &X_{s+1}[s]\ar[d]\ar[r]&P_{s}[s]\ar@{=}[d] \\
P_{s}[s-1] \ar[r]& T_{s-1}\ar[d]\ar[r] &T_{s}\ar[d]\ar[r]&P_{s}[s] \\
     & X_0\ar[d]\ar@{=}[r]& X_0\ar[d]  \\
     & X_{s}[s]\ar[r]& X_{s+1}[s+1]  }$
\end{center}where $T_s\in\mathcal{P}\ast\mathcal{P}[1]\ast\cdots\ast\mathcal{P}[s]$ for $0\leq s\leq n$. Since $M$ is tilting, $\mathrm{sup}T_s\leq0$ and $\mathrm{inf}T_{s}\geq-s$, it follows from the exact triangle $X_{-i}[-i-1]\rightarrow T_{-i-1}\rightarrow X\rightarrow X_{-i}[-i]$ that $\mathrm{inf}X_{-i}[-i-1]=i+1$, so $X_{-i}\in\mathcal{H}_M$
as $\mathrm{sup}X_{-i}[-i-1]\leq i+1$. If $n=-i$, then $X_{-i}=X_n\in\mathcal{P}$ and $\mathrm{projdim}_{\mathcal{H}_M}X_{-i}=0=n+i$. If $n>-i$, then $0\rightarrow X_{s+1}\rightarrow P_s\rightarrow X_s\rightarrow 0$ is exact in $\mathcal{H}_M$ for $s\geq-i$, so the sequence $0\rightarrow P_n\rightarrow\cdots\rightarrow P_{-i}\rightarrow X_{-i}\rightarrow0$ implies that $\mathrm{projdim}_{\mathcal{H}_M}X_{-i}=n+i$. Thus $\mathrm{FPD}(\mathcal{T},M)=\mathrm{FPD}(\mathcal{H}_M)$, the equalities follow from (2).
\end{proof}

Let $R$ be a commutative noetherian ring and $\mathcal{T}$ be $R$-linear. Following \cite{BIK1}, the category $\mathcal{T}$ is called \emph{noetherian} if  the $R$-module $\bigoplus_{n\in\mathbb{Z}}\mathrm{Hom}_\mathcal{T}(C,C[n])$ is finitely generated for any $C\in\mathcal{T}^c$. Equivalently, the $R$-module $\bigoplus_{n\in\mathbb{Z}}\mathrm{Hom}_\mathcal{T}(M,M[n])$ is finitely generated.

\begin{lem}\label{lem4.5} Let $\mathcal{T}$ be $R$-linear and noetherian. Then $\mathcal{I}$ is closed under coproducts in $\mathcal{T}$.
\end{lem}
\begin{proof} There is a homomorphism of rings $R\rightarrow\mathrm{End}(M)$. As $\mathcal{T}$ is noetherian, it follows from \cite[P.131, Exercise 7]{AF} that  $\mathrm{End}(M)$ is a noetherian ring. Let $\{I_\lambda\}_{\lambda\in\Lambda}$ be a family of objects in $\mathcal{I}$. Consider the canonical exact triangle $$\coprod_{\lambda\in\Lambda}I_\lambda\rightarrow\prod_{\lambda\in\Lambda}I_\lambda\rightarrow I\rightarrow\coprod_{\lambda\in\Lambda}I_\lambda[1].$$
 As
 $\mathrm{Hom}_\mathcal{T}(M,\coprod_{\lambda\in\Lambda}I_\lambda)
\cong\coprod_{\lambda\in\Lambda}\mathrm{Hom}_\mathcal{T}(M,I_\lambda)
\cong\coprod_{\lambda\in\Lambda}\mathrm{Hom}_\mathcal{T}(M,\mathcal{H}^0_M(I_\lambda))$ is an injective $\mathrm{End}(M)$-module, it follows from \cite[Lemma 9.1]{B} that
$$0\rightarrow\mathrm{Hom}_\mathcal{T}(M,\coprod_{\lambda\in\Lambda}I_\lambda)
\rightarrow\mathrm{Hom}_\mathcal{T}(M,\prod_{\lambda\in\Lambda}I_\lambda)\rightarrow \mathrm{Hom}_\mathcal{T}(M,I)\rightarrow0$$ is split, so $\coprod_{\lambda\in\Lambda}I_\lambda$ is a retract of $\prod_{\lambda\in\Lambda}I_\lambda$ and $\coprod_{\lambda\in\Lambda}I_\lambda\in\mathcal{I}$.
\end{proof}

\begin{lem}\label{lem3.6} Assume that $\mathcal{T}$ is  $R$-linear and noetherian. Let $\{X_\lambda\}_{\lambda\in\Lambda}$ be a collection of objects in $\mathcal{T}$ and set $X=\coprod_{\lambda\in\Lambda}X_\lambda$. If $\mathrm{injdim}_{\mathcal{T}}X_\lambda\leq n$ for all $\lambda\in\Lambda$ and $\mathrm{inf}X>-\infty$, then
$\mathrm{injdim}_{\mathcal{T}}X\leq n$.
\end{lem}
\begin{proof} By shifting if necessary, we may assume that
$\mathrm{inf}X=\mathrm{inf}\{\mathrm{inf}X_\lambda\hspace{0.03cm}|\hspace{0.03cm}\lambda\in\Lambda\} = 0$. We prove the result by induction on $n\geq0$. If $n=0$, then
$\mathrm{inf}X_\lambda=\mathrm{injdim}_\mathcal{T}X_\lambda=0$ for all $\lambda\in\Lambda$, so that $X_\lambda\in\mathcal{I}$, and then $X\in\mathcal{I}$  by Lemma \ref{lem4.5}. Suppose now that $n > 0$. For each $\lambda\in\Lambda$, choose an element $I_\lambda\in \mathcal{I}$ and a map $f_\lambda:X_\lambda\rightarrow I_\lambda$
in $\mathcal{T}$ such that $\mathcal{H}^0_M(f_\lambda)$ is injective. If $\mathrm{inf}X_\lambda>0$, then
 take $I_\lambda=0$. The map $f_\lambda$ induces an exact triangle
$X_\lambda\rightarrow I_\lambda\rightarrow L_\lambda\rightarrow X_\lambda[1]$
in $\mathcal{T}$, such that $\mathrm{inf}L_\lambda\geq0$. Let $H\in\mathcal{H}_M$ and $i\in\mathbb{Z}$.
Applying the functor $\mathrm{Hom}_\mathcal{T}(H,-)$ to the above triangle gives the following exact
sequence:
$$\mathrm{Hom}_\mathcal{T}(H[i],I_\lambda)\rightarrow\mathrm{Hom}_\mathcal{T}(H[i],L_\lambda)
\rightarrow\mathrm{Hom}_\mathcal{T}(H[i-1],X_\lambda)\rightarrow
\mathrm{Hom}_\mathcal{T}(H[i-1],I_\lambda).$$
If $i< 0$, then $\mathrm{Hom}_\mathcal{T}(H[i],I_\lambda)=0=\mathrm{Hom}_\mathcal{T}(H[i-1],I_\lambda)$ by Proposition \ref{lem3.3}, it yields that
$$\mathrm{Hom}_\mathcal{T}(H[i],L_\lambda)
\cong\mathrm{Hom}_\mathcal{T}(H[i-1],X_\lambda),$$
so $\mathrm{injdim}_\mathcal{T}L_\lambda\leq\mathrm{injdim}_\mathcal{T}X_\lambda-1$. If $\mathrm{inf}\{\mathrm{inf}L_\lambda\hspace{0.03cm}|\hspace{0.03cm}\lambda\in\Lambda\} = 0$, then  $\mathrm{injdim}_\mathcal{T}\coprod_{\lambda\in\Lambda}L_\lambda\leq n-1$ by the induction. If $\mathrm{inf}\{\mathrm{inf}L_\lambda\hspace{0.03cm}|\hspace{0.03cm}\lambda\in\Lambda\}=t> 0$, then $\mathrm{inf}\{\mathrm{inf}(L_\lambda[t])\hspace{0.03cm}|\hspace{0.03cm}\lambda\in\Lambda\}=0$.
Since $\mathrm{injdim}_\mathcal{T}L_\lambda[t]=\mathrm{injdim}_\mathcal{T}L_\lambda-t\leq n-1-t$, it follows from
 the induction that $\mathrm{injdim}_\mathcal{T}\coprod_{\lambda\in\Lambda}L_\lambda=
 \mathrm{injdim}_\mathcal{T}\coprod_{\lambda\in\Lambda}L_\lambda[t]+t\leq n-1$. Hence the exact triangle
$$\coprod_{\lambda\in\Lambda}X_\lambda\stackrel{f}\rightarrow \coprod_{\lambda\in\Lambda}I_\lambda\rightarrow \coprod_{\lambda\in\Lambda}L_\lambda\rightarrow \coprod_{\lambda\in\Lambda}X_\lambda[1]$$implies that
$\mathrm{injdim}_\mathcal{T}X=\mathrm{injdim}_\mathcal{T}\coprod_{\lambda\in\Lambda}L_\lambda+1\leq n$ by the dual of Remark \ref{lem:2.1}(4) as $f$ is an ifij morphism.
\end{proof}

Bass \cite[Proposition 4.3]{HB} showed that if $R$ is a left noetherian ring with $\mathrm{injdim}_RR< \infty$ then
$\mathrm{FPD}(R)\leq\mathrm{injdim}_RR< \infty$.
The next result generalizes this statement to noetherian triangulated categories, which is a generalization of \cite[Theorem 3.5]{LS}.

\begin{cor}\label{lem3.2'} If $\mathcal{T}$ is $R$-linear and noetherian, then $\mathrm{FPD}(\mathcal{T},M)\leq\mathrm{injdim}_\mathcal{T}M=\mathrm{silp}\mathcal{T}$ for any set $\Lambda$. The equality holds when $\mathrm{injdim}_\mathcal{T}M<\infty$ and $\mathrm{silp}\mathcal{T}:=\mathrm{sup}\{\mathrm{projdim}_\mathcal{T}I
\hspace{0.03cm}|\hspace{0.03cm}I\in\mathcal{I}\}$.
\end{cor}
\begin{proof}
As $M$ is a retract of $M^{(\Lambda)}$, $\mathrm{injdim}_\mathcal{T}M\leq\mathrm{injdim}_\mathcal{T}M^{(\Lambda)}$, it follows from Lemma \ref{lem3.6} that $\mathrm{injdim}_\mathcal{T}M^{(\Lambda)}=\mathrm{injdim}_\mathcal{T}M$, so the inequality  follows from Theorem \ref{lem3.2}. Conversely, let $n=\mathrm{injdim}_\mathcal{T}M<\infty$. Then for all $I\in\mathcal{I}$ and $i<-n$,  $\mathrm{Hom}_\mathcal{T}(I[i],M^{(\Lambda)})=0$ by Proposition \ref{lem3.3}, it follows from Proposition \ref{lem2.3} that $\mathrm{projdim}_\mathcal{T}I\leq n$. Also $\mathrm{Hom}_\mathcal{T}(J[-n],M^{(\Lambda)})\neq0$ for some $J\in\mathcal{I}$, so $n=\mathrm{projdim}_\mathcal{T}J\leq\mathrm{projdim}_\mathcal{T}J+\mathrm{inf}J
\leq\mathrm{FPD}(\mathcal{T},M)$, as desired.
\end{proof}

\begin{rem}\label{lem:0.8}{\rm (1) The first inequality in Theorem \ref{lem3.2} is obtained by Chen-Chen-Zhang in
\cite[Proposition 4.8]{CCZ} whenever $M\in\mathcal{T}^\mathrm{b}$.

(2) For a ring $R$, $\mathrm{FPD}(\mathrm{D}(R),R)=\mathrm{FPD}(R)$ the big finitistic dimension of $R$.

(3) Let $A$ be a non-positive dg ring  with bounded cohomology. Then
$\mathrm{FPD}(\mathrm{D}(A),A)=\mathrm{FPD}(A)$ the finitistic projective dimension of $A$ defined in \cite[Definition 1.1]{LS}.

(4) Let $R$ be a commutative noetherian ring and $D$ a dualizing complex over $R$ with
$\mathrm{sup}D<0$, and let
$A=R\ltimes D$ be the trivial extension. Then $\mathrm{FPD}(\mathrm{D}(A),A)\leq\mathrm{injdim}_{A}A<\infty$ by Corollary \ref{lem3.2'} and Remark \ref{lem:2.2}(6), see \cite[Theorem 4.5]{LS}.}
\end{rem}

Following \cite{CCZ,N0},
the full subcategory of \emph{pseudo-compact} objects of $\mathcal{T}$ is $$\mathcal{T}^-_c:
=\bigcap_{m\geq1}(\mathcal{T}^c\ast\mathcal{T}^{\leq-m}),$$ that is, $X\in\mathcal{T}^-_c$ if there is an exact triangle $E\rightarrow X\rightarrow D\rightharpoonup E[1]$ in $\mathcal{T}$ with $E\in\mathcal{T}^c$ and $D\in\mathcal{T}^{\leq-m}$ for any $m\geq1$, and denote $\mathcal{T}^\mathrm{b}_c:=\mathcal{T}^-_c\cap\mathcal{T}^\mathrm{b}$. We next turn to a discussion of the relationships among
$\mathrm{fpd}(\mathcal{T},M)$, $\mathrm{findim}(\mathcal{T}^c,M)$ and $\mathrm{fpd}(\mathcal{H}_M)$ introduced earlier.

\begin{thm}\label{lem4.2} One has the following $($in$)$equalities
 \begin{center}$\begin{aligned}\mathrm{fpd}(\mathcal{T},M)&
 =\mathrm{sup}\{\mathrm{projdim}_\mathcal{T}X+\mathrm{inf}X
\hspace{0.03cm}|\hspace{0.03cm}
X\in\mathcal{T}^c\cap\mathcal{T}^\mathrm{b}\}\\
&=\mathrm{sup}\{\mathrm{projdim}_\mathcal{T}X+\mathrm{inf}X
\hspace{0.03cm}|\hspace{0.03cm}
X\in\mathcal{T}^\mathrm{b}_c\ \textrm{with}\ \mathrm{projdim}_\mathcal{T}X<\infty\}\\
&=\mathrm{sup}\{\mathrm{projdim}_\mathcal{T}X
\hspace{0.03cm}|\hspace{0.03cm}
X\in\mathcal{T}^-_c\cap\mathcal{T}^{\geq0}\ \textrm{with}\ \mathrm{projdim}_\mathcal{T}X<\infty\}\\
&\leq\mathrm{findim}(\mathcal{T}^c,M)\\
&\leq
 \mathrm{max}\{0,\mathrm{fpd}(\mathcal{T},M)\}.\end{aligned}$\end{center}Moreover, if $M$ is tilting then $\mathrm{fpd}(\mathcal{T},M)=\mathrm{fpd}(\mathcal{H}_M)=\mathrm{findim}(\mathcal{T}^c,M)$.
\end{thm}
\begin{proof}  Clearly,
 $\mathcal{T}^c\cap\mathcal{T}^\mathrm{b}\subseteq\{X\in\mathcal{T}^\mathrm{b}_c
\hspace{0.03cm}|\hspace{0.03cm}\mathrm{projdim}_\mathcal{T}X<\infty\}$ by Proposition \ref{lem2.5}. Let $X\in\mathcal{T}^\mathrm{b}_c$ with $\mathrm{projdim}_\mathcal{T}X=n<\infty$. Set $s=\mathrm{sup}X$. Then $n+s+1>0$, so there is an exact triangle $E\rightarrow X[s]\rightarrow D\rightarrow E[1]$ in $\mathcal{T}$ with $E\in\mathcal{T}^c$ and $D\in\mathcal{T}^{\leq-(n+s+1)}$. As $\mathrm{projdim}_\mathcal{T}X=n$ and $\mathrm{sup}D\leq-(n+s+1)$, $\mathrm{Hom}_\mathcal{T}(X[s],D)=0$, so $X$ is a retract of $E[-s]$. Thus $\mathcal{T}^c\cap\mathcal{T}^\mathrm{b}=\{X\in\mathcal{T}^\mathrm{b}_c
\hspace{0.03cm}|\hspace{0.03cm}\mathrm{projdim}_\mathcal{T}X<\infty\}$.
Note that $\mathcal{T}^\mathrm{b}_c\cap\mathcal{T}^{\geq0}
=\mathcal{T}^-_c\cap\mathcal{T}^{\geq0}$, the equalities hold.

Let $X\in\mathcal{T}^c\cap\mathcal{T}^\mathrm{b}$ and $d=\mathrm{findim}(\mathcal{T}^c,M)$. Then $X[\mathrm{inf}X]\in\langle M\rangle^{[0,\infty)}[d]=\overline{\langle M\rangle}^{[0,\infty)}[d]\cap\mathcal{T}^c$ by \cite[Lemma 2.3]{CCZ}, it follows from Theorem \ref{lem3.2} that $\mathrm{projdim}_\mathcal{T}X[\mathrm{inf}X]\leq d$, the first inequality holds. Set $n=\mathrm{max}\{0,\mathrm{fpd}(\mathcal{T},M)\}$ and let $X\in\mathcal{T}^c\cap\mathcal{T}^{\geq0}$. Then $X\in\mathcal{T}^c\cap\overline{\langle M\rangle}^{[0,\infty)}[n]=\langle M\rangle^{[0,\infty)}[n]$ by Theorem \ref{lem3.2} and \cite[Lemma 2.3]{CCZ}, it yields the second inequality.

Let $M$ be tilting
and $X\in\mathcal{T}^\mathrm{b}_c$ with $\mathrm{projdim}_\mathcal{T}X=n<\infty$. Then $X\in\mathcal{T}^c$,  and we can choose a sppj resolution of $X$, $$X_{s+1}\rightarrow P_s\rightarrow X_s\rightarrow X_{s+1}[1]$$ such that $P_s\in\mathrm{add}(M)[-\mathrm{sup}X]$, where $X_0=X$, $X_{n+1}=0$ and $X_{n-1}\not\in\mathrm{add}(M)[-\mathrm{sup}X]$ by Remark \ref{lem:0.1}(1). The remaining statements can now be obtained by analogy with the proof of Theorem \ref{lem3.2}.
\end{proof}

Krause \cite{K1} introduced a different notion of finitistic dimensions of triangulated categories by \cite[Remark 4]{K1}. We recall this definition in detail.

For objects $X, Y$ in $\mathcal{T}$ and
$n \geq 0$, we write $h(X,Y) \leq n$ when for any $i,j\in\mathbb{Z}$,
$$\operatorname{Hom}_\mathcal{T}(X,Y[i]) \neq 0 \neq \operatorname{Hom}_\mathcal{T}(X,Y[j])
\implies |j-i| < n.$$ An object $X$ is called \emph{homologically finite} if $h(X,Y)<\infty$ for each object $Y$ in $\mathcal{T}$.
We set
$\operatorname{hom}^n(X) := \{ Y \in \mathcal{T}\hspace{0.03cm}|\hspace{0.03cm}h(X,Y) \leq n \}$.
The \emph{amplitude} of $X\in\mathcal{T}$ is defined as
$\operatorname{amp}(X):=\mathrm{sup}\{|n|\geq0\hspace{0.03cm}|\hspace{0.03cm}\mathrm{Hom}_\mathcal{T}(X,X[-n])\neq 0\}$.  An object \( X \) in \( \mathcal{T}^c \) is called  a \emph{finitistic generator} of \( \mathcal{T}^c \) if \( X \) is homologically finite and
$\hom^p(X) \subseteq \langle X\rangle_{p}\ \text{for all}\ p \geq 0$.
The \emph{finitistic dimension} of \( \mathcal{T}^c \) is defined as
\[
\operatorname{fin.dim}\mathcal{T}^c:= \inf \{ \operatorname{amp}(X)\hspace{0.03cm}|\hspace{0.03cm}X \text{ is a finitistic generator of } \mathcal{T}^c \}.
\]

The following result show that $\mathrm{fpd}(\mathcal{T},M)$ is closely related to $\operatorname{fin.dim} \mathcal{T}^c$ and generalizes \cite[Theorem 3.1]{K1}.

\begin{thm}\label{lem4.13} One has $\operatorname{fin.dim} \mathcal{T}^c\leq\mathrm{fpd}(\mathcal{T},M)$. If $M$ is tilting, then
$$\mathrm{fpd}(\mathcal{T},M)<\infty\Longleftrightarrow
\operatorname{fin.dim} \mathcal{T}^c<\infty.$$
\end{thm}
\begin{proof} We may assume that $d =\mathrm{fpd}(\mathcal{T},M)< \infty$. We show by induction on $p \ge 0$ that
$\operatorname{hom}^p(M) \subseteq \langle M\rangle_{p+d}$.
The case $p = 0$ is clear by \cite[Lemma 2.1(1)]{K1}.
Assume $p > 0$. Let $X \in \operatorname{hom}^p(M)\cap\mathcal{T}^c$ and
$\operatorname{Hom}_\mathcal{T}(M,X[n])=0$ for
$n\not\in[0,p-1]$. Then $X\cong\tau^{\geq 0}(\tau^{\leq p-1}(X))$. If $p = 1$, then $X\cong\mathcal{H}^{0}_M(X)$ with $\mathrm{projdim}_{\mathcal{T}}X=\mathrm{projdim}_{\mathcal{T}}\mathcal{H}^{0}_M(X)\leq d$, and hence $X \in\langle M\rangle_{d+1}$ by Proposition \ref{lem2.5}. For $p > 1$, as $X\in\mathrm{add}(M)[-p+1]\ast\cdots\ast\mathrm{add}(M)[-1]\ast\mathrm{add}(M)$ by Proposition \ref{lem2.5} and \cite[Lemma 2.3]{CCZ}, one has an exact triangle
\[
P\stackrel{f}\rightarrow X\rightarrow X'\rightarrow P[1]
\]in $\mathcal{T}^c$, where $P\in\mathrm{add}(M)[-p+1]$ and $X'\in\mathrm{add}(M)[-p+2]\ast\cdots\ast\mathrm{add}(M)$, it follows from Lemma \ref{lem0.2} that $\operatorname{Hom}_\mathcal{T}(M,X'[n])=0$ for
$n\not\in[0, p-2]$. By the induction,
$X'\in \operatorname{hom}^{p-1}(M) \subseteq\langle M\rangle_{p-1+d}$
and then $X \in \langle M\rangle_{p+d}$. By \cite[Lemma 2.1(4)]{K1},
$\operatorname{hom}^{p+d}(M\oplus M[-d]) = \operatorname{hom}^p(M)$, so $\operatorname{hom}^{p+d}(M\oplus M[-d])\subseteq\langle M\rangle_{p+d}= \langle M\oplus M[-d]\rangle_{p+d}$
for $p \ge 0$. Also $\operatorname{hom}^d(M\oplus M[-d]) =\operatorname{hom}^0(M)= 0$ by \cite[Lemma 2.1(4)]{K1}. Thus $M\oplus M[-d]$ is a finitistic generator of $\mathcal{T}^c$ with $\operatorname{amp}(M\oplus M[-d]) = d$, and then $\operatorname{fin.dim} \mathcal{T}^c\leq d$.

Suppose that $M$ is tilting, and $X$ is a finitistic generator of $\mathcal{T}^c$. As $X\in\langle M\rangle$ and $M\in\langle X\rangle$, it follows from
 \cite[Lemma 2.2]{K1} that
$\operatorname{hom}^1(M) \subseteq \operatorname{hom}^p(X)$ and
$\langle X\rangle_{1} \subseteq \langle M\rangle_{q}$
for some $p, q \ge 0$.
Then
$\operatorname{hom}^1(M) \subseteq \operatorname{hom}^p(X) \subseteq \langle X\rangle_{p} \subseteq\langle M\rangle_{p \cdot q}$, it implies that $\mathrm{fpd}(\mathcal{T},M)\leq p \cdot q$ by Theorem \ref{lem4.2}, Proposition \ref{lem2.5} and \cite[Corollary 3.6]{AT}.
\end{proof}

An object $H$ in $\mathcal{H}_M$ is called \emph{finitely generated} if there is an epimorphism $\mathcal{H}^0_M(M)^n\rightarrow H$ in $\mathcal{H}_M$ for some $n\geq1$. Let $R$ be a commutative noetherian ring and $\mathcal{T}$ be $R$-linear. Following \cite{KPV},  the subcategory of \emph{bounded finite} objects of $\mathcal{T}$ is denoted by
\begin{center}$\mathcal{T}^\mathrm{b}_\mathrm{f}=\{X\in\mathcal{T}\hspace{0.03cm}|\hspace{0.03cm}
\bigoplus_{n\in\mathbb{Z}}\mathrm{Hom}_\mathcal{T}(M,X[n])\in\mathrm{mod}R\}$.\end{center}

\begin{rem}\label{lem:3.8}{\rm
(1) For a ring $R$, $\mathrm{fpd}(\mathrm{D}(R),R)=\mathrm{fpd}(R)$ the finitistic dimension of $R$ by Theorem \ref{lem4.2}, and
$\mathrm{fpd}(R)<\infty$ if and only if $\mathrm{fin.dim}\operatorname{Perf}(R)<\infty$, see \cite[Theorem 3.1]{K1}.

(2) Let $\mathcal{T}$ be $R$-linear and noetherian. By \cite[Remark 5.4]{KPV} and \cite[Lemma 5.7]{Y1},
 \begin{center}$\mathcal{T}^\mathrm{b}_\mathrm{f}
 =\{X\in\mathcal{T}^\mathrm{b}\hspace{0.03cm}|\hspace{0.03cm}\mathcal{H}^\ell_M(X)\ \textrm{is\ finitely\ generated\ in}\ \mathcal{H}_M\ \textrm{for\ all}\  \ell\in\mathbb{Z}\}=\mathcal{T}^\mathrm{b}_c$.\end{center}

(3) Let $A$ be a noetherian dg ring with $\mathrm{amp}A<\infty$. Then $\mathrm{D}(A)^\mathrm{b}_\mathrm{f}=\mathrm{D}^{\mathrm{b}}_\mathrm{f}(A)$  the full subcategory of $\mathrm{D}(A)$
consisting of dg modules with finite and bounded cohomology by (2). Thus $\mathrm{fpd}(\mathrm{D}(A),A)=\mathrm{fpd}(A)$ by \cite[Definition 7.1]{BSSW}.

(4) The inequalities in Theorem \ref{lem4.2} are proved by
\cite[Proposition B.3]{BCRPZ} when $M\in\mathcal{T}^\mathrm{b}$.}
\end{rem}

\bigskip
\section {\bf Dimension behaviour under recollements and applications}
This section systematically analyzes the behavior of the (big) finitistic  and global dimensions under recollements. We begin with the following definition.

\begin{df} {\rm (\cite{BBD}) Let \( \mathcal{R}, \mathcal{S} \) and \( \mathcal{T} \) be triangulated categories. A \emph{recollement} of \( \mathcal{S} \) by \( \mathcal{R} \) and \( \mathcal{T} \) is a diagram of six triangle functors
\begin{align}
\xymatrix@C=25pt@R=5pt{
\mathcal{ R}\ar[rr]|{i_\ast=i_!} &&\mathcal{S}\ar@<-1.5ex>[ll]_{i^\ast}\ar@<+1.5ex>[ll]^{i^!} \ar[rr]|{{j^!=j^\ast}} & &\mathcal{T}\ar@<-1.5ex>[ll]_{j_!} \ar@<+1.5ex>[ll]^{j_\ast} }
\tag{$\star$}\end{align}
satisfying the following conditions:
\begin{itemize}
    \item[(R1)] \((i^*, i_*), (i_!, i^!), (j_!, j^!), (j^*, j_*)\) are adjoint pairs;
    \item[(R2)] \(i_*, j_*, j_!\) are fully faithful;
    \item[(R3)] \(j^* i_* = 0\) (thus \(i^* j_! = 0\) and \(i^! j_* = 0\));
    \item[(R4)] for any object \(Y\) in \(\mathcal{S}\), there are two exact triangles in \(\mathcal{S}\):
    \[
    i_! i^!(Y) \longrightarrow Y \longrightarrow j_* j^*(Y) \longrightarrow i_! i^!(Y)[1],
    \]
    \[
    j_! j^!(Y) \longrightarrow Y \longrightarrow i_* i^*(Y) \longrightarrow j_! j^!(Y)[1].
    \]
\end{itemize}where the maps are given by adjunctions.
The diagram consisting of the upper two rows
\begin{center}$\xymatrix@C=25pt@R=5pt{
\mathcal{ R}\ar[rr]|{i_\ast=i_!} &&\mathcal{S}\ar@<-1.5ex>[ll]_{i^\ast} \ar[rr]|{{j^!=j^\ast}} & &\mathcal{T}\ar@<-1.5ex>[ll]_{j_!} }$\end{center}
is said to be a \emph{left recollement} of \(\mathcal{S}\) by \(\mathcal{R}\) and \(\mathcal{T}\) if the four functors \(i^*, i_*, j_!\) and \(j^*\) satisfy the conditions in (R1)--(R4) involving them. A \emph{right recollement}
 is defined similarly via the lower two rows.}
\end{df}

The following theorem establishes explicit mutual bounds for FPD when the categories in the recollement admit compact silting generators.

\begin{thm}\label{lem6.1} Given a recollement $(\star)$ of compactly generated triangulated categories \( \mathcal{R}, \mathcal{S} \) and \( \mathcal{T} \) with silting compact generators $M_\mathcal{R},M_\mathcal{S},M_\mathcal{T}$.\\
$(1)$ Suppose that $i_*(M_\mathcal{R})\in \mathcal{S}^{\mathrm{sb}}$. Then

\hspace{0.2cm} $(i)$ $\operatorname{FPD}(\mathcal{S},M_\mathcal{S})\leq\operatorname{FPD}(\mathcal{R},M_\mathcal{R})+
\operatorname{FPD}(\mathcal{T},M_\mathcal{T})-\mathrm{inf}M_\mathcal{S}
+\operatorname{projdim}_{\mathcal{S}}i_*(M_\mathcal{R})+\mathrm{sup}i_\ast(M_\mathcal{R})
+\mathrm{projdim}_\mathcal{S}j_!(M_\mathcal{T})+\mathrm{sup}j_!(M_\mathcal{T})
+1$.

\hspace{0.2cm} $(ii)$ $\operatorname{FPD}(\mathcal{R},M_\mathcal{R})\leq\operatorname{FPD}(\mathcal{S},M_\mathcal{S})
+\mathrm{projdim}_\mathcal{R}i^*(M_{\mathcal{S}})+\mathrm{sup}i^*(M_{\mathcal{S}})$.\\
$(2)$ If $j_!(\mathcal{T}^{\mathrm{sb}}\cap\mathcal{T}^{\geq0})\subseteq\mathcal{S}^{\geq-e}$ for some integer $e$, then $\operatorname{FPD}(\mathcal{T},M_\mathcal{T})\leq
\operatorname{FPD}(\mathcal{S},M_\mathcal{S})+\mathrm{sup}j_!(M_\mathcal{T})+e$.
\end{thm}
\begin{proof} (1) As $i_*(M_\mathcal{R})\in \mathcal{S}^{\mathrm{sb}}$, it follows by \cite[Theorem 3.7]{CCZ} that $(\star)$ induces a left recollement
 \begin{center}$\xymatrix@C=25pt@R=5pt{
\mathcal{R}^{\mathrm{sb}}\ar[rr]|{i_\ast=i_!} &&\mathcal{S}^{\mathrm{sb}}\ar@<-1.5ex>[ll]_{i^\ast} \ar[rr]|{{j^!=j^\ast}} & &\mathcal{T}^{\mathrm{sb}}\ar@<-1.5ex>[ll]_{j_!} }$\end{center} and a right recollement
\begin{center}$\xymatrix@C=25pt@R=8pt{
\mathcal{R}^\mathrm{b}\ar[rr]|{i_\ast=i_!} &&\mathcal{S}^\mathrm{b}\ar@<+1.5ex>[ll]^{i^!} \ar[rr]|{{j^!=j^\ast}} & &\mathcal{T}^\mathrm{b}\ar@<+1.5ex>[ll]^{j_\ast} }$\end{center}
Set $\mathcal{P}_\mathcal{S}=\mathrm{Add}(M_\mathcal{S})$ and $\mathcal{P}_\mathcal{R}=\mathrm{Add}(M_\mathcal{R})$.
For $t\in\mathbb{Z}$ and $P\in\mathcal{P}_\mathcal{S}$ and $Q\in\mathcal{P}_\mathcal{R}$,
 $$\mathrm{Hom}_\mathcal{R}(M_\mathcal{R},i^!(P)[t])\cong
 \mathrm{Hom}_\mathcal{S}(i_*(M_\mathcal{R}),P[t]),$$ $$\mathrm{Hom}_\mathcal{T}(M_\mathcal{T},j^\ast(P)[t])
\cong\mathrm{Hom}_\mathcal{S}(j_!(M_\mathcal{T}),P[t]),$$
 $$\mathrm{Hom}_\mathcal{S}(M_\mathcal{S},i_*(Q))[t])
\cong\mathrm{Hom}_\mathcal{R}(i^*(M_\mathcal{S}),Q[t]),$$
As $i_*$ and $j^*$ have right adjoints, it follows from Proposition \ref{lem2.3} and Lemma  \ref{lem0.2} that $$a:=\mathrm{sup}\{\mathrm{sup}i^!(P)
 \hspace{0.03cm}|\hspace{0.03cm}
 P\in\mathcal{P}_\mathcal{S}\}=\operatorname{projdim}_{\mathcal{S}}i_*(M_\mathcal{R})<\infty,$$ $$b:=\mathrm{sup}j^*(M_\mathcal{S})=\mathrm{sup}\{\mathrm{sup}j^*(P)
 \hspace{0.03cm}|\hspace{0.03cm}
 P\in\mathcal{P}_\mathcal{S}\}=\mathrm{projdim}_\mathcal{S}j_!(M_\mathcal{T})<\infty,$$
 $$c:=\mathrm{sup}i_*(M_\mathcal{R})=\mathrm{sup}\{\mathrm{sup}i_*(Q)\hspace{0.03cm}|\hspace{0.03cm}
Q\in\mathcal{P}_\mathcal{R}\}=\mathrm{projdim}_\mathcal{R}i^*(M_\mathcal{S})<\infty.$$

 (i) If either $\operatorname{FPD}(\mathcal{R},M_\mathcal{R})=\infty$ or $\operatorname{FPD}(\mathcal{T},M_\mathcal{T})=\infty$ or $\mathrm{inf}M_\mathcal{S}=-\infty$, then (i) is trivial. We may assume
$m:=\operatorname{FPD}(\mathcal{R},M_\mathcal{R})<\infty$ and $n:=\operatorname{FPD}(\mathcal{T},M_\mathcal{T})<\infty$ and $M_\mathcal{S}\in \mathcal{S}^b$.
Let $Y \in \mathcal{S}^b \cap \mathcal{S}^{\geq 0}$ with $\operatorname{projdim}_{\mathcal{S}}Y < \infty$. There is a canonical exact triangle in $\mathcal{S}$
\begin{align}
j_! j^!(Y) \to Y \to i_* i^*(Y) \to j_! j^!(Y)[1].
\label{exact04}\tag{$\sharp$}\end{align}
As $Y \in \mathcal{S}^{\mathrm{sb}}$, $j^!(Y)\in \mathcal{T}^{\mathrm{sb}}$, and so
$\operatorname{projdim}_{\mathcal{T}}j^!(Y)< \infty$.
Since \( Y \in \mathcal{S}^{\geq 0} \), it follows that
\[
\text{Hom}_{\mathcal{T}}(M_{\mathcal{T}}[p],j^!(Y)) \cong \text{Hom}_\mathcal{S}(j_!(M_{\mathcal{T}})[p],Y) = 0 \text{ for } p >\mathrm{sup}j_!(M_\mathcal{T}).
\]Hence  $j^!(Y)\in \mathcal{T}^{-}\cap\mathcal{T}^{\geq -\mathrm{sup}j_!(M_\mathcal{T})}$, and hence $\operatorname{projdim}_{\mathcal{T}}j^!(Y)=
\operatorname{projdim}_{\mathcal{T}}j^!(Y)[-\mathrm{sup}j_!(M_\mathcal{T})]+
\mathrm{sup}j_!(M_\mathcal{T})\leq n+\mathrm{sup}j_!(M_\mathcal{T})$. For $P\in\mathcal{P}_\mathcal{S}$, one has\[
\text{Hom}_{\mathcal{S}}(j_!j^!(Y),P[t]) \cong \text{Hom}_\mathcal{T}(j^!(Y),j^!(P)[t]) = 0 \text{ for } t >n+\mathrm{sup}j^!(P)+\mathrm{sup}j_!(M_\mathcal{T}),
\]it implies that $\operatorname{projdim}_{\mathcal{S}}j_!j^!(Y)\leq n+b+\mathrm{sup}j_!(M_\mathcal{T})$ by Proposition \ref{lem2.3}, and $j_!j^!(Y)\in\mathcal{S}^{\geq \mathrm{inf}M_\mathcal{S}-n-b-\mathrm{sup}j_!(M_\mathcal{T})}$
by Remark \ref{lem:2.1}(2).
As \( Y \in \mathcal{S}^{\geq 0} \), it follows from the triangle ($\sharp$) that $i_*i^*(Y)\in\mathcal{S}^{\geq \mathrm{inf}M_\mathcal{S}-n-b-\mathrm{sup}j_!(M_\mathcal{T})-1}$.
As $i_*$ is fully faithful, $$\text{Hom}_\mathcal{R}(M_{\mathcal{R}}[p],i^*(Y))
\cong\text{Hom}_\mathcal{S}(i_*(M_{\mathcal{R}})[p],i_*i^*(Y))=0$$ for $p>n-\mathrm{inf}M_\mathcal{S}+\mathrm{sup}i_\ast(M_\mathcal{R})
+\mathrm{sup}j_!(M_\mathcal{T})+b+1$. Then $\mathrm{inf}i^*(Y)\geq-n+\mathrm{inf}M_\mathcal{S}-\mathrm{sup}i_\ast(M_\mathcal{R})
-\mathrm{sup}j_!(M_\mathcal{T})-b-1$, so $\operatorname{projdim}_{\mathcal{R}}i^*(Y)\leq m-\mathrm{inf}i^*(Y)\leq m+n-\mathrm{inf}M_\mathcal{S}+\mathrm{sup}i_\ast(M_\mathcal{R})
+\mathrm{sup}j_!(M_\mathcal{T})+b+1$. As $i^!(P)\in\mathcal{R}^-$ for $P\in\mathcal{P}_\mathcal{S}$  by \cite[Theorem 3.7(3)]{CCZ},
$$\text{Hom}_{\mathcal{S}}(i_*i^*(Y),P[t]) \cong \text{Hom}_\mathcal{R}(i^*(Y),i^!(P)[t]) = 0 \text{ for } t >\mathrm{projdim}_\mathcal{R}i^*(Y)+\mathrm{sup}i^!(P),$$so $\operatorname{projdim}_{\mathcal{S}}i_*i^*(Y)\leq
m+n-\mathrm{inf}M_\mathcal{S}+a+\mathrm{sup}i_\ast(M_\mathcal{R})
+b+\mathrm{sup}j_!(M_\mathcal{T})
+1$. Thus $\operatorname{projdim}_\mathcal{S}Y\leq\operatorname{projdim}_{\mathcal{S}}i_*i^*(Y)$ by the triangle ($\sharp$), as claimed.

(ii)
We may assume $l:=\operatorname{FPD}(\mathcal{S},M_\mathcal{S})<\infty$.
Let $X\in \mathcal{R}^\mathrm{b} \cap \mathcal{R}^{\geq 0}$ with $\operatorname{projdim}_{\mathcal{R}}X < \infty$. Then $i_\ast(X)\in \mathcal{S}^{\mathrm{sb}}\cap\mathcal{S}^{\mathrm{b}}$. As
$\text{Hom}_{\mathcal{S}}(M_{\mathcal{S}},i_*(X)[t]) \cong \text{Hom}_\mathcal{R}(i^*(M_{\mathcal{S}}),X[t]) = 0 \text{ for }t <-\mathrm{sup}i^\ast(M_\mathcal{S})$, it implies that $\mathrm{inf}i_*(X)\geq-\mathrm{sup}i^\ast(M_\mathcal{S})$. For $Q\in\mathcal{P}_\mathcal{R}$, as
$$\text{Hom}_{\mathcal{R}}(X,Q[t]) \cong \text{Hom}_\mathcal{S}(i_*(X),i_*(Q)[t]) = 0 \text{ for } t >\mathrm{projdim}_\mathcal{S}i_*(X)+\mathrm{sup}i_*(Q),$$  one has $\mathrm{projdim}_\mathcal{R}X\leq\mathrm{projdim}_\mathcal{S}i_*(X)+
c\leq
l-\mathrm{inf}i_*(X)
+c\leq l+c+\mathrm{sup}i^*(M_{\mathcal{S}})$, it yields the inequality in (ii).

(2)
Let $Z\in \mathcal{T}^\mathrm{b} \cap \mathcal{T}^{\geq 0}$ with $\operatorname{projdim}_{\mathcal{T}}Z < \infty$. As $\text{Hom}_\mathcal{S}(j_!(Z),P[t])\cong
\text{Hom}_{\mathcal{T}}(Z,j^!(P)[t])= 0 \text{ for } P\in\mathcal{P}_\mathcal{S}\ \text{and}\ t >\mathrm{projdim}_\mathcal{T}Z+\mathrm{sup}j^!(P)$, it implies that $\mathrm{projdim}_\mathcal{S}j_!(Z)\leq\operatorname{projdim}_{\mathcal{T}}Z
+b$. Set $\mathcal{P}_\mathcal{T}:=\mathrm{Add}(M_\mathcal{T})$. As $j_!$ is fully faithful,
$\text{Hom}_{\mathcal{T}}(Z,W[t]) \cong \text{Hom}_\mathcal{S}(j_!(Z),j_!(W)[t]) = 0 \text{ for } W\in\mathcal{P}_\mathcal{T}\ \text{and}\ t >\mathrm{projdim}_\mathcal{S}j_!(Z)+\mathrm{sup}j_!(W)$. Since $(j_!,j^!)$ is an adjoint pair and $Z\in \mathcal{T}^\mathrm{sb} \cap \mathcal{T}^{\geq 0}$, $\mathrm{sup}j_!(W)\leq\mathrm{sup}j_!(M_\mathcal{T})$, it yields that $\mathrm{projdim}_\mathcal{T}Z\leq\mathrm{projdim}_\mathcal{S}j_!(Z)+\mathrm{sup}j_!(M_\mathcal{T})\leq \operatorname{FPD}(\mathcal{S},M_\mathcal{S})-\mathrm{inf}j_!(Z)
+\mathrm{sup}j_!(M_{\mathcal{T}})\leq
\operatorname{FPD}(\mathcal{S},M_\mathcal{S})+\mathrm{sup}j_!(M_\mathcal{T})+e$, as claimed.
\end{proof}

\begin{cor}\label{lem6.3} Let \( R, S, T \) be rings such that there exists a recollement:
\begin{center}$\xymatrix@C=25pt@R=8pt{
\mathrm{D}(R)\ar[rr]|{i_\ast=i_!} &&\mathrm{D}(S)\ar@<-1.5ex>[ll]_{i^\ast}\ar@<+1.5ex>[ll]^{i^!} \ar[rr]|{{j^!=j^\ast}} & &\mathrm{D}(T)\ar@<-1.5ex>[ll]_{j_!} \ar@<+1.5ex>[ll]^{j_\ast} }$\end{center}
$(1)$ Suppose that $\mathrm{projdim}_{S}i_*(R)<\infty$. Then

\hspace{0.2cm} $(i)$ If \( \operatorname{FPD}(R) < \infty \) and \( \operatorname{FPD}(T) < \infty \), then \( \operatorname{FPD}(S) < \infty \).

\hspace{0.2cm} $(ii)$ If \( \operatorname{FPD}(S) < \infty \), then \( \operatorname{FPD}(R) < \infty \).\\
$(2)$ Suppose that $\mathrm{inf}\{\mathrm{inf}j_!(H)\hspace{0.03cm}|\hspace{0.03cm}H\in\mathrm{Mod}T\}\geq-e$ for some $e\in\mathbb{Z}$. If \( \operatorname{FPD}(S) < \infty \), then \( \operatorname{FPD}(T) < \infty \).
\end{cor}
\begin{proof}  By \cite[Theorem 2.2]{AKL}, $\mathrm{sup}j_!(T)$ is compact in $\mathrm{D}(S)$. By \cite[Lemma 2.9(e)]{AKLY}, $i^*(S)$ is compact in $\mathrm{D}(R)$. The statements follows from Theorem \ref{lem6.1} and Remark \ref{lem:0.8}(2).
\end{proof}

Recall that a ring epimorphism \(\lambda : R \to S\) is \emph{homological} if \(\operatorname{Tor}_i^R(S,S) = 0\) for all \(i > 0\).

\begin{cor}\label{lem6.4} Let \( R \) be a ring and \( e \) an idempotent element of \( R \). Suppose that the canonical surjection
\(\pi: R \to R/ReR \) is homological with \(\mathrm{projdim}_{R}ReR<\infty \). Then
\[ \operatorname{FPD}(R/ReR) \leq \operatorname{FPD}(R)  \leq \operatorname{FPD}(eRe) + \operatorname{FPD}(R/ReR) + \operatorname{projdim}_{R}R/ReR+ 1.\]
\end{cor}
\begin{proof}
As \(\pi: R \to R/ReR \) is homological, there is a recollement
\begin{center}$\xymatrix@C=25pt@R=8pt{
\mathrm{D}(R/ReR)\ar[rr]|{D(\pi_\ast)} &&\mathrm{D}(R)\ar@<-1.5ex>[ll]_{\ \ \ R/ReR \otimes_R^{\mathbf{L}}-}\ar@<+1.5ex>[ll]^{i^!} \ar[rr]|{eR\otimes_R^{\mathbf{L}}-} & &\mathrm{D}(eRe)\ar@<-1.5ex>[ll]_{Re\otimes_{eRe}^{\mathbf{L}}-} \ar@<+1.5ex>[ll]^{j_\ast} }$\end{center}As $D(\pi_\ast)(R/ReR)=R/ReR$ and $\mathrm{projdim}_{R}ReR<\infty$,
\(\mathrm{projdim}_{R}D(\pi_\ast)(R/ReR)<\infty \). The corollary follows from Theorem \ref{lem6.1}(1)  and Remark \ref{lem:0.8}(2).
\end{proof}

\begin{cor}\label{lem6.6} {\rm (\cite[Proposition 2.16]{JYZ}).} Let $R$ be a commutative ring, \( A \) a non-positive dg \( R \)-algebra and \( e \in A \) an idempotent. If $A$ is K-flat as an $R$-complex, then there is a non-positive dg \( R \)-algebra \( B \)  and a recollement of derived categories
\begin{center}$\xymatrix@C=25pt@R=8pt{
\mathrm{D}(B)\ar[rr]|{i_\ast=i_!} &&\mathrm{D}(A)\ar@<-1.5ex>[ll]_{i^\ast}\ar@<+1.5ex>[ll]^{i^!} \ar[rr]|{{j^!=j^\ast}} & &\mathrm{D}(eAe)\ar@<-1.5ex>[ll]_{j_!} \ar@<+1.5ex>[ll]^{j_\ast} }$\end{center}
If $\mathrm{projdim}_{eAe}eA<\infty$, then
 \( \operatorname{FPD}(B)\leq\operatorname{FPD}(A)\leq\operatorname{FPD}(eAe)+\operatorname{FPD}(B)+\operatorname{projdim}_AB+1\).
\end{cor}

We now establish  mutual bounds for fpd under recollements.

\begin{lem}\label{lem4.8} Let
 $X\in\mathcal{T}^\mathrm{b}_c$ with $\mathrm{projdim}_\mathcal{T}X<\infty$. Then $$\mathrm{projdim}_\mathcal{T}X=\mathrm{sup}\{n\in\mathbb{Z}\hspace{0.03cm}|\hspace{0.03cm}
\mathrm{Hom}_\mathcal{T}(X,M[n])\neq0\}.$$
\end{lem}
\begin{proof} Let $\mathrm{projdim}_\mathcal{T}X=n$. Then $X\in\mathcal{T}^c$ by the proof of Theorem \ref{lem4.2} and $\mathrm{Hom}_\mathcal{T}(X,H[n])\neq0$ for some $H\in\mathcal{H}_M$. Consider the exact triangle $H'\rightarrow M^{(\Lambda)}\stackrel{f}\rightarrow H\rightarrow H'[1]$ with $f$ a sppj morphism and $\Lambda$ a set. As $\mathrm{sup}H'\leq0$, one has the following exact sequence $$\mathrm{Hom}_\mathcal{T}(X,M^{(\Lambda)}[n])\rightarrow \mathrm{Hom}_\mathcal{T}(X,H[n])\rightarrow
\mathrm{Hom}_\mathcal{T}(X,H'[n+1])=0,$$so $\mathrm{Hom}_\mathcal{T}(X,M[n])\neq0$. Conversely, let $\mathrm{sup}\{n\in\mathbb{Z}\hspace{0.03cm}|\hspace{0.03cm}
\mathrm{Hom}_\mathcal{T}(X,M[n])\neq0\}=n$. Then $\mathrm{projdim}_\mathcal{T}X=m\geq n$. If $m>n$, then $\mathrm{Hom}_\mathcal{T}(X,M[m])\neq0$ by the preceding proof, a contradiction. Thus $\mathrm{projdim}_\mathcal{T}X=n$.
\end{proof}

The following theorem provides analogous inequalities for fpd under recollements.

\begin{thm}\label{lem7.1} Given a recollement $(\star)$ of compactly generated triangulated categories \( \mathcal{R}, \mathcal{S} \) and \( \mathcal{T} \) with silting compact generators $M_\mathcal{R},M_\mathcal{S},M_\mathcal{T}$.\\
$(1)$ Suppose that $i_*(M_\mathcal{R})\in \mathcal{S}^{c}$. Then

\hspace{0.2cm} $(i)$ $\operatorname{fpd}(\mathcal{S},M_\mathcal{S})\leq\operatorname{fpd}(\mathcal{R},M_\mathcal{R})+
\operatorname{fpd}(\mathcal{T},M_\mathcal{T})-\mathrm{inf}M_\mathcal{S}
+\mathrm{projdim}_\mathcal{S}i_\ast(M_\mathcal{R})+\mathrm{sup}i_\ast(M_\mathcal{R})
+\mathrm{projdim}_\mathcal{S}j_!(M_\mathcal{T})+\mathrm{sup}j_!(M_\mathcal{T})
+1$.

\hspace{0.2cm} $(ii)$ $\operatorname{fpd}(\mathcal{R},M_\mathcal{R})\leq\operatorname{fpd}(\mathcal{S},M_\mathcal{S})
+\mathrm{projdim}_\mathcal{R}i^\ast(M_\mathcal{S})+\mathrm{sup}i^*(M_{\mathcal{S}})$.\\
$(2)$ If $j_!(\mathcal{T}^c\cap\mathcal{T}^{\geq0})\subseteq\mathcal{S}^{\geq-e}$ for some $e\in\mathbb{Z}$, then $\operatorname{fpd}(\mathcal{T},M_\mathcal{T})\leq
\operatorname{fpd}(\mathcal{S},M_\mathcal{S})+\mathrm{sup}j_!(M_\mathcal{T})+e$.
\end{thm}
\begin{proof}  As $i_*(M_\mathcal{R})\in \mathcal{S}^{c}$, it follows by \cite[Lemma 2.9(2)]{CCZ} that $(\star)$ induces a left recollement:
 \begin{center}$\xymatrix@C=25pt@R=5pt{
\mathcal{R}^{c}\ar[rr]|{i_\ast=i_!} &&\mathcal{S}^{c}\ar@<-1.5ex>[ll]_{i^\ast} \ar[rr]|{{j^!=j^\ast}} & &\mathcal{T}^{c}\ar@<-1.5ex>[ll]_{j_!} }$\end{center}
Now, arguing analogously to the proof of Theorem \ref{lem6.1}, the desired inequalities hold.
\end{proof}

Let \( R, S \) and \( T \) be rings with identity, and let \( \lambda : R \to S \) and \( \mu : R \to T \) be ring homomorphisms.
Suppose that \( M \) is an \( S \)-\( T \)-bimodule together with an element \( m \in M \). We say that the quadruple \( (\lambda, \mu, M, m) \) is an \emph{exact context} if
\[
0 \longrightarrow R \xrightarrow{(\lambda, \mu)} S \oplus T \xrightarrow{\begin{pmatrix} \cdot m \\ -m\cdot \end{pmatrix}} M \longrightarrow 0
\]
is an exact sequence of abelian groups, where \( \cdot m \) and \( m \cdot \) denote the right and left multiplication, respectively. Given an exact context \( (\lambda, \mu, M, m) \), there is a ring with identity, called the \emph{noncommutative tensor product} of \( (\lambda, \mu, M, m) \) and denoted by \( T \boxtimes_R S \), see \cite{CX}.

\begin{cor}\label{lem7.5} Let \((\lambda, \mu, M, m)\) be an exact context with \(\operatorname{Tor}_i^R(T, S) = 0\) for all \(i \geq 1\). If \( {_R}S\) has a finite projective resolution by finitely generated projective modules, then

\hspace{0.2cm} $(i)$ \(\operatorname{fpd}(T \boxtimes_R S) \leq \operatorname{fpd}(S) + \operatorname{fpd}(T) + 1\).

\hspace{0.2cm} $(ii)$ \(\operatorname{fpd}(S) \leq \operatorname{fpd}\begin{pmatrix} S & M \\ 0 & T \end{pmatrix} \leq \operatorname{fpd}(R) + \operatorname{fpd}(T \boxtimes_R S) + \max\{1, \operatorname{projdim}_RS\} + 3\).
\end{cor}
\begin{proof}
(i) Let \( B := \begin{pmatrix} S & M \\ 0 & T \end{pmatrix},\ C := M(T \otimes_R S) \). As \( \operatorname{Tor}_i^R(T, S) = 0 \) for all \( i \geq 1 \), there exists a recollement:
\begin{center}$\xymatrix@C=25pt@R=8pt{
\mathrm{D}(C)\ar[rr]|{i_\ast} &&\mathrm{D}(B)\ar@<-1.5ex>[ll]_{\ \ \ C \otimes_B^{\mathrm{L}}-}\ar@<+1.5ex>[ll] \ar[rr]|{j^!} & &\mathrm{D}(R)\ar@<-1.5ex>[ll]_{j_!} \ar@<+1.5ex>[ll]}$\end{center} (see \cite{CX0} for details). By \cite[Corollary 5.8(1)]{CX0},
  $\operatorname{projdim}_Bi_\ast(C)\leq \max\{2, \operatorname{projdim}_RS + 1\}$ and \( C \otimes_B^{\mathrm{L}} B \cong C \) in \( \mathcal{D}(C) \), it follows from Theorem \ref{lem7.1} that \( \operatorname{fpd}(T \boxtimes_R S)=\operatorname{fpd}(C) \leq \operatorname{fpd}(B)\leq \operatorname{fpd}(S) + \operatorname{fpd}(T) + 1 \).

(ii) As $\operatorname{projdim}_Bj_!(R)=1$ and $\operatorname{sup}j_!(R)=0$, it follows from Theorem \ref{lem7.1} that
$\operatorname{fpd}(S) \leq\operatorname{fpd}(B) \leq \operatorname{fpd}(R) + \operatorname{fpd}(T \boxtimes_R S) + \max\{2, \operatorname{projdim}_RS + 1\} + 1 + 1$.
\end{proof}

Given a ring \( R \) and an \( R \)-\( R \)-bimodule \( M \), the \emph{trivial extension}, \( R \ltimes M \), of \( R \) by \( M \) is a ring  with abelian group \( R \oplus M \) and multiplication:
\[
(r, m)(r', m') = (rr', rm' + mr') \quad \text{for } r, r' \in R \text{ and } m, m' \in M.
\]

\begin{cor}\label{lem7.6} Let \( \lambda : R \to S \) be a ring epimorphism and \( M \) an \( S \)-\( S \)-bimodule such that \( \operatorname{Tor}_i^R(M, S) = 0 \) for all \( i \geq 1 \). If \( {_R}S \) has a finite projective resolution by finitely generated projective \( R \)-modules, then

$(a)$ \( \operatorname{fpd}(S \ltimes M) \leq \operatorname{fpd}(S) + \operatorname{fpd}(R \ltimes M) + 1 \).

$(b)$ \( \operatorname{fpd}(S) \leq \operatorname{fpd}(R) + \operatorname{fpd}(S \ltimes M)  + \max\{1, \operatorname{projdim}_RS\} + 3\).
\end{cor}
\begin{proof}
Let \( T := R \ltimes M, \mu : R \to T \) be the inclusion. Then the pair \( (\lambda, \mu) \) is exact (see \cite{CX} for details),
so \( T \boxtimes_R S \simeq S \ltimes M \) as rings. As \( \operatorname{Tor}_i^R (T, S)\cong \operatorname{Tor}_i^R (M, S) = 0 \) for all \( i \geq 1 \), the statements follows from Corollary \ref{lem7.5}.
\end{proof}

The next theorem provides analogous inequalities for global dimensions.

\begin{thm}\label{lem6.2} Given a recollement $(\star)$ of compactly generated triangulated categories \( \mathcal{R}, \mathcal{S} \) and \( \mathcal{T} \) with silting compact generators $M_\mathcal{R},M_\mathcal{S},M_\mathcal{T}$.

$(a)$ $\operatorname{gpd}(\mathcal{S},M_\mathcal{S})\leq\operatorname{gpd}(\mathcal{R},M_\mathcal{R})+
\operatorname{gpd}(\mathcal{T},M_\mathcal{T})-\mathrm{inf}M_\mathcal{S}+
\mathrm{projdim}_\mathcal{S}i_\ast(M_\mathcal{R})+\mathrm{sup}i_\ast(M_\mathcal{R})
+\mathrm{projdim}_\mathcal{S}j_!(M_\mathcal{T})+\mathrm{sup}j_!(M_\mathcal{T})
+1$.

$(b)$ If  $M_\mathcal{S}\in \mathcal{S}^{\mathrm{b}}$, then $\operatorname{gpd}(\mathcal{R},M_\mathcal{R})\leq\operatorname{gpd}(\mathcal{S},M_\mathcal{S})
+\mathrm{sup}i_*(M_\mathcal{R})+\mathrm{sup}i^*(M_{\mathcal{S}})$.

$(c)$ If  $M_\mathcal{R}\in \mathcal{R}^{\mathrm{b}}$, then $\operatorname{gpd}(\mathcal{T},M_\mathcal{T})\leq
\operatorname{gpd}(\mathcal{S},M_\mathcal{S})+\mathrm{sup}j_\ast(M_\mathcal{T})+\mathrm{sup}j^\ast(M_\mathcal{S})$.
\end{thm}
\begin{proof} $(a)$ We may assume that $\operatorname{gpd}(\mathcal{R},M_\mathcal{R})<\infty$ and $\operatorname{gpd}(\mathcal{T},M_\mathcal{T})<\infty$ and $M_\mathcal{S}\in\mathcal{S}^\mathrm{b}$. Then $M_\mathcal{T}\in \mathcal{T}^{\mathrm{b}}$ by
\cite[Lemma 3.3]{CCZ}, and so $j^\ast(M_\mathcal{S})\in\mathcal{T}^{\mathrm{b}}=\mathcal{T}^{\mathrm{sb}}$ by \cite[Corollary 3.12]{BNP} and Remark \ref{lem:2.2}(4). Hence \cite[Theorem 3.7]{CCZ} implies that $i_*(M_\mathcal{R})\in \mathcal{S}^{\mathrm{sb}}$. Let $Y \in \mathcal{S}^\mathrm{b} \cap \mathcal{S}^{\geq 0}$. Then $j^!(Y)\in \mathcal{T}^{\mathrm{b}}$, and so $j^!(Y)\in \mathcal{T}^{\mathrm{sb}}$. Consequently,
 $j_!j^!(Y)\in\mathcal{S}^{\mathrm{sb}}\subseteq\mathcal{S}^{\mathrm{b}}$
by \cite[Proposition 3.4]{CCZ}, and hence $i_*i^*(Y)\in\mathcal{S}^{\mathrm{b}}$. Now, by analogy with the proof of Theorem \ref{lem6.1}(i), one has $\operatorname{projdim}_\mathcal{S}Y<\infty$, so $\operatorname{gpd}(\mathcal{S},M_\mathcal{S})=\mathrm{FPD}(\mathcal{S},M_\mathcal{S})$, as claimed.

$(b)$  We may assume that $\operatorname{gpd}(\mathcal{S},M_\mathcal{S})<\infty$.
Let $X\in \mathcal{R}^\mathrm{b} \cap \mathcal{R}^{\geq 0}$. Then $i_\ast(X)\in \mathcal{S}^{\mathrm{b}}=\mathcal{S}^{\mathrm{sb}}$ by Remark \ref{lem:2.2}(4) as $M_\mathcal{S}\in \mathcal{S}^{\mathrm{b}}$. An analogous argument to that used in the proof of Theorem \ref{lem6.1}(ii), one has $\mathrm{projdim}_\mathcal{R}X<\infty$, it yields the inequality in (b).

$(c)$ We may assume that $\operatorname{gpd}(\mathcal{S},M_\mathcal{S})<\infty$. As  $M_\mathcal{R}\in \mathcal{R}^{\mathrm{b}}$, $i_\ast(M_\mathcal{R})\in\mathcal{S}^{\mathrm{sb}}$ by \cite[Corollary 3.12]{BNP} and Remark \ref{lem:2.2}(4). Let $Z\in \mathcal{T}^\mathrm{b} \cap \mathcal{T}^{\geq 0}$. Then $j_\ast(Z)\in \mathcal{S}^\mathrm{b}$ by \cite[Theorem 3.7]{CCZ} and $\mathrm{projdim}_\mathcal{S}j_\ast(Z)<\infty$. Since $\text{Hom}_\mathcal{S}(M_\mathcal{S},j_\ast(Z)[t])\cong
\text{Hom}_{\mathcal{T}}(j^\ast(M_\mathcal{S}),Z[t])= 0 \text{ for } t <\mathrm{inf}Z-\mathrm{sup}j^\ast(M_\mathcal{S})$, $\mathrm{inf}j_\ast(Z)\geq-\mathrm{sup}j^\ast(M_\mathcal{S})$.
As $j_\ast$ is fully faithful,
$$\text{Hom}_{\mathcal{T}}(Z,M_\mathcal{T}[t]) \cong \text{Hom}_\mathcal{S}(j_\ast(Z),j_\ast(M_\mathcal{T})[t]) = 0\ \textrm{for}\ t >\mathrm{projdim}_\mathcal{S}j_\ast(Z)+\mathrm{sup}j_\ast(M_\mathcal{T}),$$ it yields that $\mathrm{projdim}_\mathcal{T}Z\leq\mathrm{projdim}_\mathcal{S}j_\ast(Z)+\mathrm{sup}j_\ast(M_\mathcal{T})
\leq \operatorname{gpd}(\mathcal{S},M_\mathcal{S})-\mathrm{inf}j_\ast(Z)
+\mathrm{sup}j_\ast(M_{\mathcal{T}})\leq
\operatorname{gpd}(\mathcal{S},M_\mathcal{S})+
\mathrm{sup}j_\ast(M_\mathcal{T})+\mathrm{sup}j^\ast(M_\mathcal{S})$, as required.
\end{proof}

\begin{cor}\label{lem8.7}  Given a recollement $(\star)$ of compactly generated triangulated categories \( \mathcal{R}, \mathcal{S} \) and \( \mathcal{T} \) with silting compact generators $M_\mathcal{R},M_\mathcal{S},M_\mathcal{T}$.
Let $M_\mathcal{R}\in \mathcal{R}^{\mathrm{b}}$ and $M_\mathcal{S}\in \mathcal{S}^{\mathrm{b}}$. Then $$\operatorname{gpd}(\mathcal{S},M_\mathcal{S})<\infty\Longleftrightarrow \operatorname{gpd}(\mathcal{R},M_\mathcal{R})<\infty\ \textrm{and}\ \operatorname{gpd}(\mathcal{T},M_\mathcal{T})<\infty.$$
\end{cor}

\bigskip \centerline {\bf Acknowledgements}
This research was partially supported by National Natural Science Foundation of China (12571035, 12671053).

\end{document}